\documentclass[11pt]{article}
\usepackage[]{amsmath,amssymb,amsthm,amsfonts,amsbsy, enumerate,eucal}
\usepackage{verbatim}
\usepackage{hyperref}
\usepackage{array}
\usepackage{xfrac}
\usepackage{color}
\usepackage{mathtools}
\usepackage{ulem} 
\usepackage[T1]{fontenc}
\usepackage[english]{babel}
\usepackage{graphicx}

\usepackage{cite}
\usepackage{tikz}

\theoremstyle{definition}

\newtheorem{definition}{Definition}[section]
\newtheorem{example}[definition]{Example}

\newtheorem{remark}[definition]{Remark}

\newtheorem{note}[definition]{Note}

\newtheorem{notation}[definition]{Notation}

\newcommand\scalemath[2]{\scalebox{#1}{\mbox{\ensuremath{\displaystyle #2}}}}

\theoremstyle{plain}
\newtheorem{theorem}[definition]{Theorem}
\newtheorem{lemma}[definition]{Lemma}

\newtheorem{proposition}[definition]{Proposition}

\usepackage[top=3cm, bottom=3cm, left=2.5cm, right=2.5cm]{geometry}

\def\I{\mathbb I}

\def\C{\mathbb C}

\def\H{\hat{H}_{q}}
\def\Hqi{\hat{H}_{q^{-1}}}

\def\MX{{\rm Mat}_{X}(\mathbb{C})}

\def\Mx{M\hat{x}}
\def\Mxp{M\hat{x}^{\perp}}
\def\MC{M\hat{C}}

\def\p{{\perp}}
\def\<{\langle}
\def\>{\rangle}
\def\diag{{\rm diag}}
\def\bdiag{{\rm blockdiag}}

\def\X{{\bf X}}
\def\Y{{\bf Y}}
\def\A{{\bf A}}
\def\B{{\bf B}}

\def\W{{\bf W}}
\def\T{{\bf T}}
\def\k{\kappa}
\def\v0{{v^{\perp}_0}}
\def\y{{\bf y}}

\def\wt{\widetilde}

\newcommand{\Ga}{\ensuremath{\Gamma}}
\newcommand{\tht}{\ensuremath{\theta}}

\newcommand{\mcal}{\ensuremath{\mathcal}}

\begin{document}
\title{{\bf 
Nonsymmetric Askey-Wilson polynomials and 
$Q$-polynomial distance-regular graphs}}

\author{Jae-Ho Lee \\ \\
\Large{\it Dedicated to Professor Paul Terwilliger on his 60th birthday}
}

\date{}

\maketitle

\begin{abstract}

 In his famous theorem (1982), Douglas Leonard characterized the $q$-Racah polynomials and their relatives in the Askey scheme from the duality property of $Q$-polynomial distance-regular graphs.
In this paper we consider a nonsymmetric (or Laurent) version of the $q$-Racah polynomials in the above situation.
Let $\Ga$ denote a $Q$-polynomial distance-regular graph that contains a Delsarte clique $C$.
 Assume that $\Ga$ has $q$-Racah type.
 Fix a vertex $x \in C$.
 We partition the vertex set of $\Ga$ according to the path-length distance to both $x$ and $C$. 
 The linear span of the characteristic vectors corresponding to the cells in this partition has an irreducible module structure for the universal double affine Hecke algebra $\H$ of type $(C^{\vee}_1, C_1)$.
 From this module, we naturally obtain a finite sequence of orthogonal Laurent polynomials.
 We prove the orthogonality relations for these polynomials, using the $\H$-module and the theory of Leonard systems.
Changing $\H$ by $\Hqi$ we show how our Laurent polynomials are related to the nonsymmetric Askey-Wilson polynomials, and therefore how our Laurent polynomials can be viewed as nonsymmetric $q$-Racah polynomials.


 \end{abstract}
\section{Introduction}\label{Intro}

The nonsymmetric Askey-Wilson polynomials were first treated by Sahi \cite{SS}.
They are expressed as certain Laurent polynomials, and 
are obtained in the double affine Hecke algebra (DAHA)
of type $(C^{\vee}_1, C_1)$ as eigenfunctions of the Cherednik-Dunkl operator 
on the basic representation for the algebra. The nonsymmetric
Askey-Wilson polynomials along with the DAHA of rank one were
studied in algebraic aspects by Noumi and Stokman \cite{NS}, Macdonald \cite[Section 6.6]{IM} and Koornwinder \cite{TK,KB}. 
In the present paper we study the nonsymmetric Askey-Wilson polynomials 
in a combinatorial aspect, using a $Q$-polynomial distance-regular graph that contains a Delsarte 
clique.

\medskip
The $Q$-polynomial property for distance-regular graphs was introduced by Delsarte \cite{PD}. 
Since then, this property has been receiving substantial attention from many mathematicians; see e.g. \cite{BI, BCN, CG,DL, DKT, PT:T-alg}.
In \cite{DL}, Leonard characterized the $q$-Racah polynomials and their relatives in the Askey scheme using the duality property of $Q$-polynomial distance-regular graphs (see also \cite[Section~III.5]{BI}). 
Terwilliger defined the subconstituent algebra (or Terwilliger algebra) as a 
method of the study of $Q$-polynomial distance-regular graphs \cite{PT:T-alg, PT;T-alg;II, PT;T-alg;III}. 
This algebra has been a significant tool in the study of $Q$-polynomial distance-regular graphs and its connections to Lie theory, quantum algebras, 
and coding theory have also been revealed; see e.g. \cite{ITT,IT2,IT4,AS,GST,IT3,HT}. 

\medskip
In \cite{JHL} the author showed a relationship between $Q$-polynomial distance-regular graphs and the universal DAHA $\H$ of type $(C^{\vee}_1,C_1)$  using the 
Terwilliger algebra. We briefly summarize this result.
Let $\Ga$ denote a $Q$-polynomial distance-regular graph that contains a Delsarte clique $C$. Assume that $\Ga$ has $q$-Racah type. Fix a vertex $x \in C$. Partitioning the vertex set of $\Ga$ according to the path-length distance to both $x$ and $C$ gives a two-dimensional equitable partition, which takes a staircase shape consisting of nodes and edges. Let $\W$ denote the $\C$-vector space spanned by the characteristic vectors corresponding to the nodes of the staircase shape of the partition. 
Then $\W$ has an irreducible module structure for the algebra $\H$ \cite[Sections 11,12]{JHL}.

\medskip
From the above staircase picture of $\W$, the $q$-Racah polynomials \cite{AW2} arise naturally as follows.
Roughly speaking, horizontal edges correspond to a sequence of $q$-Racah polynomials and vertical edges correspond to another sequence of $q$-Racah polynomials. 
In the present paper, using the irreducible $\H$-module $\W$, we define a finite sequence of certain Laurent polynomials that correspond to nodes of the staircase picture.
We denote these polynomials by $\varepsilon^{\sigma}_i$, where $\sigma \in \{+,-\}$. 
The $\varepsilon^{\sigma}_i$ are considered as nonsymmetric $q$-Racah polynomials, a discrete version of the nonsymmetric Askey-Wilson polynomials.
And then we treat the orthogonality relations for $\varepsilon^{\sigma}_i$,
using the $\H$-module $\W$ and the theory of Leonard systems \cite{PT:Madrid}. 
This orthogonality is new and it can be viewed as a discrete version of the orthogonality for the nonsymmetric Askey-Wilson polynomials, which was worked by Koornwinder and Bouzeffour \cite[Section 5]{KB}.

\medskip
The paper is organized as follows. 
In Section 2, we review some basic definitions, concepts and notation
regarding nonsymmetric Askey-Wilson polynomials and DAHAs of type $(C^{\vee}_1, C_1)$.
In Sections 3 and 4, we review some backgrounds concerning 
$Q$-polynomial distance-regular graphs, the Terwilliger algebra, Leonard systems, and parameter arrays. 
Our vector space $\W$ appears along with a comprehensible picture in Section 3.
In Section 5, we study the module for the Terwilliger algebra $T$ on $\W$ and the associated $q$-Racah polynomials. 
The $T$-module $\W$ decomposes into the direct sum of two irreducible $T$-modules, and the Leonard system corresponding to each $T$-module gives rise to a sequence of the $q$-Racah polynomials. We express these polynomials and the related formulae in terms of certain scalars $a,b,c,d$.
In Section 6, we recall the algebra $\H$ and its properties. And we display the $\H$-module $\W$ in terms of the scalars $a,b,c,d$. For this module, we describe  the action of $\X:=t_3t_0 \in \H$.

\medskip
In Section 7,  we define the Laurent polynomial $g$ that plays a role to connect the above two irreducible $T$-submodules of $\W$. 
Using $g$ and the $\H$-module $\W$, we define a finite sequence of Laurent polynomials $\varepsilon^{\sigma}_i (\sigma \in \{+,-\})$.
Moreover, for the element $\Y:=t_0t_1 \in \H$ we describe the action of $\varepsilon^{\sigma}_i[\Y]$ on the $\H$-module $\W$.
In Section 8, we compute the eigenvalues/ eigenvectors of $\Y$ on $\W$.
Using the results, in Section 9 we define a bilinear form on the vector space $L$ spanned by $\{\varepsilon^{\sigma}_i\}^{D-1}_{i=0}$. With respect to this bilinear form we prove the orthogonality relations for the Laurent polynomials $\varepsilon^{\sigma}_i$.
In Section 10, we consider the algebra $\Hqi$ by changing $q$ by $q^{-1}$.
We discuss how the algebra $\Hqi$ is related to the (ordinary) DAHA $\tilde{\mathfrak{H}}$ of type $(C^{\vee}_1, C_1)$. 
We, further, redescribe the Laurent polynomials $\varepsilon^{\sigma}_i$ and the associated formulae in terms of $q^{-1}$-version.
In Section 11, we make a normalization for $\varepsilon^{\sigma}_i$ of $q^{-1}$-version and discuss how these polynomials are related to the nonsymmetric Askey-Wilson polynomials.
The paper ends with a brief summary and direction for future work in Section 12.
An Appendix provides some explicit data involving the $\H$-action on $\W$.

\begin{notation}
Throughout this paper we assume $q \in \C^*$ is not a root of unity. 
For $a\in \C$, 
\begin{equation}\label{(a;q)n}
(a;q)_n:=(1-a)(1-aq)\cdots(1-aq^{n-1}),
\end{equation}
where $n=0,1,2,\ldots$. For $a_1, a_2, \ldots, a_r \in \C$, 
$$
(a_1,a_2, \ldots, a_r;q)_n := (a_1;q)_n(a_2;q)_n \cdots (a_r;q)_n.
$$
Let $\C[z,z^{-1}]$ denote the space of the Laurent polynomials with a variable $z$. 
We write an element of $\C[z,z^{-1}]$ by $f[z]$.
We say $f[z]$ is {\it symmetric} if $f[z]=f[z^{-1}]$, otherwise {\it nonsymmetric}.
Note that a symmetric Laurent polynomial $f[z]$ can be viewed as an ordinary 
polynomial $f(x)$ in $x=z+z^{-1}$.
\end{notation}

\section{Nonsymmetric Askey-Wilson polynomials}\label{NonsymAWpolys}

In this section we review some backgrounds concerning the Askey-Wilson polynomials,
DAHAs of type $(C^{\vee}_1,C_1)$, and the nonsymmetric Askey-Wilson polynomials.
For more background, see \cite{AW,TK,IM,NS}. We acknowledge that 
notation and presentations of
the nonsymmetric Askey-Wilson polynomials and 
the DAHA of type $(C^{\vee}_1,C_1)$
are taken from Koornwinder's papers \cite{TK, KB}. Throughout this section, let $a,b,c,d \in \C^*$ be such that
\begin{equation*}
\begin{split}
&ab,ac,ad,bc,bd,cd,abcd \notin \{q^{-m} \mid m=0,1,2, \ldots\}.
\end{split}
\end{equation*}
We now recall the Askey-Wilson polynomials \cite{AW}. For $n=0,1,2,\ldots$ define a polynomial 
\begin{align}
\nonumber
p_n(x)=p_n[z;a,b,c,d \mid q] & :=  \sum^{\infty}_{i=0} \frac{(q^{-n}, abcdq^{n-1}, az, az^{-1};q)_i}
{(ab, ac, ad, q;q)_i}q^i\\
\label{AW;abcd}
& = {_4}\phi_3
\left(
\begin{matrix}
q^{-n}, ~ abcdq^{n-1}, ~ az, ~  az^{-1} \\
ab, ~ ac, ~ ad
\end{matrix}
~\middle| ~ q,~q
\right),
\end{align}
where $x=z+z^{-1}$.
The last equality follows from the definition of basic hypergeometric series \cite[p. 4]{GR}. 
Observe that $(q^{-n};q)_i=0$ if $i > n$. We call $p_n$ the $n$-th {\it Askey-Wilson polynomial}. Consider the {\it monic} Askey-Wilson polynomials
\begin{equation*}
P_n = P_n[z;a,b,c,d \mid q] := \frac{(ab,ac,ad;q)_n}{a^n(abcdq^{n-1};q)_n}
{_4}\phi_3
\left(
\begin{matrix}
q^{-n}, ~ abcdq^{n-1}, ~ az, ~  az^{-1} \\
ab, ~ ac, ~ ad
\end{matrix}
~\middle| ~ q,~q
\right).
\end{equation*}
Note that $P_n$ is symmetric. For $n=1,2,\ldots$, define a Laurent polynomial \cite[Section 4]{KB}
\begin{equation}\label{nonsym;Q}
Q_n := a^{-1}b^{-1}z^{-1}(1-az)(1-bz)P_{n-1}[z;qa,qb,c,d\mid q].
\end{equation}

\begin{definition}\label{Koornwinder;NonsymAW} \cite[\S4 (4.2)--(4.3)]{KB}
The {\it nonsymmetric Askey-Wilson polynomials}  are defined by
\begin{align*}
& E_{-n} := P_n-Q_n && (n=1,2,\ldots),\\
& E_n := P_n-\frac{ab(1-q^n)(1-cdq^{n-1})}{(1-abq^n)(1-abcdq^{n-1})}Q_n && (n=0,1,2,\ldots),
\end{align*}
where $(1-q^n)Q_n := 0$ for $n=0$.
\end{definition}

\noindent
The DAHA of type $(C^{\vee}_1, C_1)$, denoted by $\tilde{\mathfrak{H}}$ \cite[Section 3]{TK}, is defined by the generators $Z, Z^{-1}, T_0, T_1$ and relations
\begin{align*}
& (T_1+ab)(T_1+1)=0, && (T_0+q^{-1}cd)(T_0+1)=0, \\
& (T_1Z+a)(T_1Z+b)=0, && (qT_0Z^{-1}+c)(qT_0Z^{-1}+d)=0.
\end{align*}
The algebra $\tilde{\mathfrak{H}}$ has a faithful representation on $\C[z,z^{-1}]$, which is called the {\it basic representation} \cite[Section 3]{TK}.
On the basic representation, by \cite[Theorem 4.1]{TK} each of $E_{\pm n}$ is the eigenfunction for $Y=T_1T_0$;
\begin{align}\label{Y-act1}
&YE_{-n} = q^{-n}E_{-n} \qquad (n=1,2, \ldots), \\
\label{Y-act2}
&YE_n = q^{n-1}abcdE_n \qquad (n=0,1,2, \ldots).
\end{align}
Fix square roots $a^{1/2}, b^{1/2}, c^{1/2}, d^{1/2}$ and $q^{1/2}$.
Consider $q^{1/2}\eta Y$, where $\eta=a^{-1/2}b^{-1/2}c^{-1/2}d^{-1/2}$. From (\ref{Y-act1}) and (\ref{Y-act2}), it follows
\begin{align}\label{q;eta;Y1}
&q^{1/2}\eta YE_{-n} = q^{-n+\frac{1}{2}}\eta E_{-n} \qquad (n=1,2, \ldots), \\
\label{q;eta;Y2}
&q^{1/2}\eta YE_n = q^{n-\frac{1}{2}}\eta^{-1}E_n \qquad (n=0,1,2, \ldots).
\end{align}
By (\ref{q;eta;Y1}) and (\ref{q;eta;Y2}), we give a staircase diagram that describes the structure of eigenspaces of 
$q^{1/2}\eta Y$.

\begin{center}
\scalemath{0.65}{
\begin{tikzpicture}
  [scale=.8,thick,auto=left, every node/.style={circle}] 
  \node[fill=black,label=right:{\Large$q^{-1/2}\eta^{-1}, E_0$}] (n1) at (0,0) {};
  \node[draw, label=left:{\Large$q^{-1/2}\eta,E_{-1}$}] (n2) at (0,2)  {};
  \node[fill=black,label=right:{\Large$q^{1/2}\eta^{-1},E_1$}] (n3) at (2,2)  {};
  \node[draw,label=left:{\Large$q^{-3/2}\eta,E_{-2}$}] (n4) at (2,4) {};
  \node[fill=black,label=right:{\Large$q^{3/2}\eta^{-1},E_2$}] (n5) at (4,4)  {};
  \node[draw,label=left:{\Large$q^{-5/2}\eta,E_{-3}$}] (n6) at (4,6)  {};
  \node[fill=black,label=right:{\Large$q^{5/2}\eta^{-1},E_3$}] (n7) at (6,6)  {};
  \node[draw,label=left:{\Large$q^{-7/2}\eta, E_{-4}$}] (n8) at (6,8)  {};
  \node[fill=black,label=right:{\Large$q^{7/2}\eta^{-1}, E_4$}] (n9) at (8,8)  {};
  \node[draw,label=left:{\Large$$}] (n10) at (8,10)  {};
  \node[label=right:{\large$$}] (n11) at (10,10)  {};  

  \foreach \from/\to in 
  {n1/n2,n2/n3,n3/n4,n4/n5,n5/n6,n6/n7,n7/n8,n8/n9}
      \draw (\from) -- (\to);
      \draw (n1) -- (n2) ;
      \draw (n3) -- (n4) ;
      \draw (n5) -- (n6) ;
      \draw (n7) -- (n8) ;
      \draw (n9) -- (n10) [dashed, ];
      \draw (n10) -- (n11) [dashed];
            
\end{tikzpicture}}\\
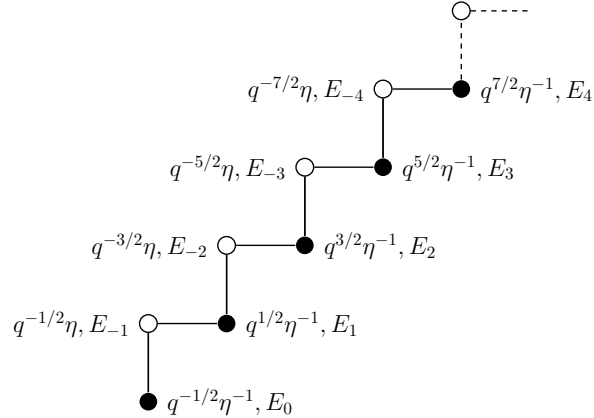
{Figure 1 : The eigenspaces of $q^{1/2}\eta Y$}
\end{center}

We remark that each white node represents the eigenspace of $q^{1/2}\eta Y$ corresponding the eigenvalue $q^{-n+\frac{1}{2}}\eta$ and the eigenvector $E_{-n}$ for $n=1,2,\ldots$, and each black node represents the eigenspace of $q^{1/2}\eta Y$ corresponding the eigenvalue $q^{n-\frac{1}{2}}\eta^{-1}$ and the eigenvector $E_n$ for $n=0,1,2,\ldots$.
Observe that the product of eigenvalues of each vertical edge is equal to $q^{-1}$ and the product of eigenvalues of each horizontal edge is equal to $1$.

\medskip
We discuss the orthogonality relations for the Askey-Wilson polynomials. 
By \cite[Theorems I.4.4 and II.3.2]{TC} (or \cite[(3.6)--(3.8)]{KB}), 
there exists a positive Borel measure 
$\mu=\mu_{a,b,c,d;q}$ on $\mathbb{R}$ with $\mu(\mathbb{R})=1$ such that
\begin{equation}\label{AW;orthogonality}
\<P_m, P_n\>_{a,b,c,d;q} := \int_{\mathbb{R}}P_m(x)P_n(x)d\mu(x)=h_n\delta_{m,n},
\end{equation}
where 
\begin{equation*}
h_n=h^{a,b,c,d;q}_n=\frac{(q,ab,ac,ad,bc,bd,cd;q)_n}{(abcd;q)_{2n}(abcdq^{n-1};q)_n}.
\end{equation*}

In \cite{KB} Koorwinder and Bouzeffour introduced a presentation of nonsymmetric Laurent polynomials as two-dimensional vector-valued polynomials. By \cite[p. 7]{KB}, we can identify a Laurent polynomial $f$ with 2-vector-valued symmetric Laurent polynomial $(f_1, f_2)^t$, where $t$ denotes transpose. In particular, from \cite[(4.10) and (4.11)]{KB}
\begin{align}
\label{E;2-vec;form;-n}
&E_{-n}=
\begin{pmatrix}
P_n[z;a,b,c,d \mid q] \\
-a^{-1}b^{-1}P_{n-1}[z;aq,bq,c,d\mid q]
\end{pmatrix}
\qquad (n=1,2,\ldots), \\
\label{E;2-vec;form;n}
& E_n = 
\begin{pmatrix}
P_n[z;a,b,c,d \mid q] \\
-\dfrac{(1-q^n)(1-cdq^{n-1})}{(1-abq^n)(1-abcdq^{n-1})}P_{n-1}[z;aq,bq,c,d\mid q]
\end{pmatrix}
\qquad (n=0,1,2,\ldots),
\end{align}
where $(1-q^n)P_{n-1} := 0$ for $n=0$.
In \cite[Section 5]{KB}, the authors introduced a symmetric bilinear form 
$\<\cdot,\cdot \>$ on $\C[z,z^{-1}]$:
\begin{equation}\label{bilinear;TK}
\<g,h\> = \<(g_1,g_2)^t,(h_1,h_2)^t\> = \<g_1,h_1\>_{a,b,c,d;q}+C\<g_2,h_2\>_{aq,bq,c,d;q},
\end{equation}
where $\< \cdot, \cdot\>_{a,b,c,d;q}$ is from (\ref{AW;orthogonality}) and
$$
C = -ab\frac{(1-ab)(1-abq)(1-ac)(1-ad)(1-bc)(1-bd)}{(1-abcd)(1-abcdq)}.
$$
Note that the nonsymmetric Askey-Wilson polynomials $E_n (n \in \mathbb{Z})$ are orthogonal with respect to the bilinear form (\ref{bilinear;TK}). This bilinear form is positive definite with some conditions for the scalars $a,b,c,d$; 
see \cite[Proposition 5.1]{KB} for details.

\begin{lemma}\label{norm;E} 
With respect to the bilinear form {\rm(\ref{bilinear;TK})},
\begin{enumerate}
\item[\rm(i)] for $n=1,2,\ldots,$
$$
\<E_{-n},E_{-n}\> = \frac{(ab-1)(1-abcdq^{2n-1})}{ab(1-q^n)(1-cdq^{n-1})}\frac{(q,ab,ac,ad,bc,bd,cd;q)_n}{(abcd;q)_{2n}(abcdq^{n-1};q)_n}.
$$
\item[\rm(ii)] for $n=0,1,2,\dots,$
$$
\<E_{n},E_{n}\> = \frac{(1-ab)(1-abcdq^{2n-1})}{(1-abq^n)(1-abcdq^{n-1})}\frac{(q,ab,ac,ad,bc,bd,cd;q)_n}{(abcd;q)_{2n}(abcdq^{n-1};q)_n}.
$$
\end{enumerate}
\end{lemma}
\begin{proof} 
(i) From (\ref{E;2-vec;form;-n}), we set $f_1=P_n[z;a,b,c,d \mid q]$ and $f_2=-a^{-1}b^{-1}P_{n-1}[z;aq,bq,c,d\mid q]$. Then by (\ref{bilinear;TK})
$$
\<E_{-n}, E_{-n}\> = \<(f_1,f_2)^t,(f_1,f_2)^t\>=\<f_1,f_1\>_{a,b,c,d;q}+C\<f_2,f_2\>_{aq,bq,c,d;q}.
$$
Compute the right-hand side of the above equation by using (\ref{AW;orthogonality}). The result follows.\\
(ii) Similar to (i).
\end{proof}

\section{$Q$-polynomial distance-regular graphs}\label{S3;DRG}

We recall some basic concepts and notation concerning $Q$-polynomial distance-regular graphs. For more information we refer to the reader to \cite{BI, BCN, PT:T-alg}. Let $X$ denote a nonempty finite set. Define $\MX$ to be the $\C$-algebra consisting of the square matrices indexed by $X$ with entries in $\C$. Let $V$ denote the $\C$-vector space consisting of column vectors indexed by $X$ with entries in $\C$. View $V$ as a left $\MX$-module. We endow $V$ with the Hermitian inner product $\< \cdot, \cdot \>_V$ such that $\<u,v\>_V=u^t\bar{v}$, where $t$ denotes transpose and $\bar{}$ denotes complex conjugate. 
We abbreviate $\Vert u \Vert^2 = \<u,u\>_V$ for all $u \in V$.
For $y \in X$ let $\hat{y}$ denote the vector in $V$ with a $1$ in the $y$-coordinate and $0$ in all other coordinates. For $Y\subseteq X$ define $\hat{Y}=\sum_{y\in Y}\hat{y}$, called the {\it characteristic vector} of $Y$.

\smallskip
Let $\Ga$ denote a simple connected graph with vertex set $X$ and diameter $D \geq 3$, where $D := \max\{\partial(x,y)\mid x,y \in X\}$ and where $\partial$ is the shortest path-length distance function. For $x \in X$, define 
\begin{equation}\label{Ga_i}
\Ga_i(x) = \{ y \in X \mid \partial(x,y)=i \} \qquad ( 0 \leq i \leq D).
\end{equation}
We say that $\Ga$ is {\it distance-regular} whenever for $0 \leq i \leq D$ and vertices $x,y \in X$ with $\partial(x,y)=i$ the numbers
\begin{equation}\label{intersection;numbers}
c_i = |\Ga_{i-1}(x)\cap \Ga_1(y)|, \qquad a_i = |\Ga_i(x) \cap \Ga_1(y)|, \qquad b_i=|\Ga_{i+1}\cap\Ga_1(y)|,
\end{equation}
are independent of $x$ and $y$. Define the matrix $A_i \in \MX$ by $(A_i)_{xy}=1$ if $\partial(x,y)=i$ and $0$ otherwise. We call $A_i$ the $i$-th {\it distance matrix} of $\Ga$. In particular, $A=A_1$ is called the {\it adjacency matrix} of $\Ga$. Let $M$ denote the subalgebra of $\MX$ generated by $A$, called the {\it adjacency algebra}. By definition, every element in $M$ forms a polynomial in $A$. The graph $\Ga$ satisfies the {\it $P$-polynomial property}, that is, for $0 \leq i \leq D$ there exists a polynomial $f_i \in \C[x]$ such that $\deg(f_i)=i$ and $f_i(A)=A_i$.

\smallskip
We recall the notion of $Q$-polynomial property.  By \cite[p.~127]{BCN}, the elements $\{A_i\}^D_{i=0}$ form a basis for $M$. Since $A$ is real symmetric and generates $M$, $A$ has $D+1$ mutually distinct real eigenvalues, denoted by $\tht_0, \tht_1, \ldots, \tht_D$. Let $E_i \in \MX$ denote the orthogonal projection onto the eigenspace of $\tht_i (0 \leq i \leq D)$. We call $E_i$ the {\it $i$-th primitive idempotent} of $\Ga$. 
Note that $\{E_i\}^D_{i=0}$ form a basis for $M$.
We say that $\Ga$ is {\it $Q$-polynomial} with respect to the ordering $E_0, E_1, \ldots E_D$ whenever there exists $f^*_i \in \C[x]$ such that $\deg(f^*_i)=i$ and $f^*_i(E_1)=E_i$, where the multiplication of $M$ is under the entrywise product \cite[p.~193]{BI}. Throughout the paper we assume that $\Ga$ is a $Q$-polynomial distance-regular graph.

\smallskip
By a {\it clique} we mean a nonempty subset $C$ of $X$ such that any two distinct vertices in $C$ are adjacent. It is known that $|C| \leq 1-k/\tht_{\rm min}$ \cite[Proposition 4.4.6]{BCN}, where $\tht_{\rm min}$ is the minimum eigenvalue of $A$. We say that $C$ is {\it Delsarte} when $|C| = 1-k/\tht_{\rm min}$. Assume that $\Ga$ contains a Delsarte clique $C$. For $0 \leq i \leq D-1$, we define 
\begin{equation}\label{Ci}
C_i := \{ y \in X \mid \partial(y,C)=i \},
\end{equation}
where $\partial(y,C) = \min\{ \partial(y,z) \mid z \in C \}$. For the rest of the paper we fix a vertex $x \in C$. Recall $\Ga_i = \Ga_i(x) ~ (0 \leq i \leq D)$ and $C_i ~ (0 \leq i \leq D-1)$ from (\ref{Ga_i}) and (\ref{Ci}). For $0 \leq i \leq D-1$ define 
\begin{equation}\label{Ci_pm}
C^-_i=\Ga_i \cap C_i, \qquad \qquad C^+_i = \Ga_{i+1}\cap C_i \qquad \qquad (0 \leq i \leq D-1).
\end{equation}

\begin{center}
\scalemath{0.65}{
\begin{tikzpicture}
  [scale=1,thick,auto=left,every node/.style={circle,draw}] 
  \node (n1) at (0,0) {$C^-_0$};
  \node (n2) at (0,2)  {$C^+_0$};
  \node (n3) at (2,2)  {$C^-_1$};
  \node (n4) at (2,4) {$C^+_1$};
  \node (n5) at (4,4)  {$C^-_2$};
  \node (n6) at (4,6)  {$C^+_2$};   
  \node (n7) at (6,6)  {$C^-_3$};
  \node (n8) at (6,8)  {$C^+_3$};
 
  \node[fill=black!10] (c0) at (0,-2.5) {${~}C^{~}_0$};
  \node[fill=black!10] (c1) at (2,-2.5) {${~}C^{~}_1$};
  \node[fill=black!10] (c2) at (4,-2.5) {${~}C^{~}_2$};
  \node[fill=black!10] (c3) at (6,-2.5) {${~}C^{~}_3$};
  
  \node[fill=blue!20] (r0) at (-2.5,0) {${~}\Ga_0$};
  \node[fill=blue!20] (r1) at (-2.5,2) {${~}\Ga_1$};
  \node[fill=blue!20] (r2) at (-2.5,4) {${~}\Ga_2$};
  \node[fill=blue!20] (r3) at (-2.5,6) {${~}\Ga_3$};
  \node[fill=blue!20] (r4) at (-2.5,8) {${~}\Ga_4$};

  \foreach \from/\to in {n1/n2,n2/n3,n3/n4,n4/n5,n5/n6,n6/n7,n7/n8,n1/n3,n3/n5,n5/n7,n2/n4,n4/n6,n6/n8, c0/c1,c1/c2,c2/c3, r0/r1, r1/r2, r2/r3, r3/r4}
    \draw (\from) -- (\to);
    
    \draw (c0) -- (n1) [dashed];
    \draw (c1) -- (n3) [dashed];
    \draw (c2) -- (n5) [dashed];
    \draw (c3) -- (n7) [dashed];
    
    \draw (r0) -- (n1) [dashed];
    \draw (r1) -- (n2) [dashed];
    \draw (r2) -- (n4) [dashed];
    \draw (r3) -- (n6) [dashed];
    \draw (r4) -- (n8) [dashed];

\end{tikzpicture}}\\
\medskip
Figure {2} : The set $\{C^{\pm}_i\}$ of $X$  when $d=4$
\end{center}

\noindent
Note that each of $C^{\pm}_i (0 \leq i \leq D-1)$ is nonempty, and by construction the $\{C^{\pm}_i\}^{D-1}_{i=0}$ is an {\it equitable} partition of $X$ in the sense of \cite[p. 75]{CG}; see \cite[Proposition 5.6]{JHL}. Define $\W$ to be the subspace of $V$ spanned by $\{\hat{C}^{\pm}_i\}^{D-1}_{i=0}$. By the previous comments one readily sees that $\{\hat{C}^{\pm}_i\}^{D-1}_{i=0}$ is an orthogonal basis for $\W$.

\medskip
For $0 \leq i \leq D$ define the diagonal matrix $E^*_i=E^*_i(x) \in \MX$ 
by $(E^*_i)_{yy} = 1$ if $\partial(x,y)=i$ and $0$ otherwise. 
We call $E^*_i$ the {\it $i$-th dual primitive idempotent} of $\Ga$ with respect to $x$.
Observe that $I = \sum^{D}_{i=0}E^*_i$ and $E^*_iE^*_j=\delta_{i,j}E^*_i$ 
for $0 \leq i,j \leq D$. So the set $\{E^*_i\}^D_{i=0}$ forms a basis for 
a commutative subalgebra $M^*=M^*(x)$ of $\MX$. We call $M^*$ 
the {\it dual adjacency algebra} of $\Ga$ with respect to $x$. 
Define the diagonal matrix $A_i^*=A_i^*(x) \in \MX$ by
$(A^*_i)_{yy} = |X|(E_i)_{xy}$ for $y \in X$, called the 
{\it $i$-th dual distance matrix} of $\Ga$ with respect to $x$. 
By \cite[p. 379]{PT:T-alg} $\{A^*_i\}^D_{i=0}$ is a basis for $M^*$. 
We abbreviate $A^*=A^*_1$, called the {\it dual adjacency matrix} of 
$\Ga$ with respect to $x$. By \cite[Lemma 3.11]{PT:T-alg} $A^*$ 
generates $M^*$. By these comments $A^*$ has $D+1$ mutually 
distinct real eigenvalues, denoted by 
$\tht^*_0, \tht^*_1, \ldots, \tht^*_D$ and called $\tht^*_i$ the 
$i$-th {\it dual eigenvalue} of $A^*$.

\smallskip
{\it Terwilliger algebra} $T=T(x)$ with respect to $x$ is the subalgebra 
of $\MX$ generated by $A, A^*$ \cite{PT:T-alg}. By $T$-module, we 
mean a subspace $W \subseteq V$ such that $BW \subseteq W$ for 
all $B \in T$. We define 
$\wt A^* = \wt A^*(C)={|C|}^{-1}\sum_{y \in C}A^*(y) \in \MX$, called 
the {\it dual adjacency matrix} of $\Ga$ with respect to $C$. 
The Terwilliger algebra $\wt{T}=\wt{T}(C)$ with respect to $C$ is the 
subalgebra of $\MX$ generated by $A, \wt{A}^*$ \cite{HS}. 
In \cite[Definition 5.20]{JHL} we defined the generalized Terwilliger 
algebra $\T=\T(x,C)$. The algebra $\T$ is the subalgebra of $\MX$ 
generated by $T, \wt T$. Observe that $A, A^*$ and $\wt{A}^*$ generate 
$\T$. Note that $\W$ has a module structure for both $T$ and $\wt T$, 
and so is a $\T$-module \cite[Proposition 5.25]{JHL}. The $T$-submodule 
(resp. $\wt{T}$-submodule) of $\W$ generated by $\hat{x}$ (resp. $\hat{C}$) 
will be called the {\it primary $T$-module} (resp. primary $\wt{T}$-module), 
denoted by $M\hat{x}$ (resp. $M\hat{C}$). The $\{A_i\hat{x}\}^D_{i=0}$ 
(resp. $\{\hat{C}_i\}^{D-1}_{i=0}$) is a basis for $M\hat{x}$ (resp. $M\hat{C}$).  
In Section \ref{T-mod;W}, we will discuss the $T$-module $\W$ in more detail.

\section{Leonard systems and parameter arrays}

Let {\sf d} denote a positive integer. Let $M_{{\sf d}+1}(\C)$ denote the $\C$-algebra consisting of all $({\sf d}+1)\times ({\sf d}+1)$ matrices that have entries in $\C$. Let $\mcal{A}$ denote a $\C$-algebra isomorphic to $M_{{\sf d}+1}(\C)$. Let $\sf V$ denote an irreducible left $\mcal{A}$-module. Remark that $\sf V$ is unique up to isomorphism of $\mathcal{A}$-modules and $\sf V$ has dimension ${\sf d}+1$. For ${\sf A} \in \mcal{A}$, ${\sf A}$ is called {\it multiplicity-free} whenever ${\sf A}$ has ${\sf d}+1$ mutually distinct eigenvalues. Assume ${\sf A}$ is multiplicity-free. Let $\{ \tht_i\}^{\sf d}_{i=0}$ denote an ordering of distinct eigenvalues of $\sf A$. For $0 \leq i \leq {\sf d}$ let ${\sf V}_i$ denote the eigenspace of $\sf A$ associated with $\tht_i$. Define ${\sf E}_i \in \mcal{A}$ by $({\sf E}_i-{\rm I}){\sf V}_i=0$ and ${\sf E}_i{\sf V}_j=0$ for $j \ne i ~(0 \leq j\leq {\sf d})$, where ${\rm I}$ is the identity of $\mcal{A}$. We call ${\sf E}_i$ the {\it primitive idempotent} of ${\sf A}$ associated with $\tht_i$. Observe that (i) ${\sf A}{\sf E}_i=\tht_i{\sf E}_i$, (ii) ${\sf E}_i{\sf E}_j=\delta_{i,j}{\sf E}_i$, (iii) $\sum^{\sf d}_{i=0}{\sf E}_i={\rm I}$. We now define a Leonard system in $\mcal{A}$.

\begin{definition}\cite[Definition 1.4]{PT;2lin}\label{Def;LS}
By a {\it Leonard system} on $\sf V$, we mean a sequence
$$
\Phi = ({\sf A}; {\sf A}^*; \{{\sf E}_i\}^{\sf d}_{i=0}; \{{\sf E}^*_i\}^{\sf d}_{i=0})
$$
that satisfies (i)--(v) below.
\begin{itemize}
\item[(i)] Each of ${\sf A}, {\sf A}^*$ is a multiplicity-free element in $\mcal{A}$.
\item[(ii)] $\{{\sf E}_i\}^{\sf d}_{i=0}$ is an ordering of the primitive idempotents of ${\sf A}$.
\item[(iii)] $\{{\sf E}^*_i\}^{\sf d}_{i=0}$ is an ordering of the primitive idempotents of ${\sf A}^*$.
\item[(iv)] For $0 \leq i,j \leq {\sf d}$, 
$$
{\sf E}_i{\sf A}^*{\sf E}_j = 
\begin{cases}
0 & \text{ if } \quad |i-j|>1,\\
\ne 0 & \text{ if } \quad |i-j| = 1.
\end{cases}
$$
\item[(v)] For $0 \leq i,j \leq {\sf d}$, 
$$
{\sf E}^*_i{\sf A}{\sf E}^*_j = 
\begin{cases}
0 & \text{ if } \quad |i-j|>1,\\
\ne 0 & \text{ if } \quad |i-j| = 1.
\end{cases}
$$
\end{itemize}
We call {\sf d} the {\it diameter} of $\Phi$, and say $\Phi$ is {\it over} $\C$.
\end{definition}

\begin{example}\label{ex;LS;Mx} 
Recall from Section \ref{S3;DRG} that $\Ga$ is a $Q$-polynomial
distance-regular graph and $T$ is the Terwilliger algebra of $\Ga$
with respect to $x$.
Referring to Section \ref{S3;DRG}, consider a sequence of elements of $T$
\begin{equation}\label{ex;LS;Mx;eq}
(A; A^*; \{E_i\}^D_{i=0}; \{E^*_i\}^D_{i=0}),
\end{equation}
where $A$ (resp. $A^*$) is the adjacency matrix (resp. dual adjacency matrix) of $\Ga$ and $E_i$ (resp. $E^*_i$) is the $i$-th primitive idempotent (resp. dual primitive idempotent) of $\Ga$. Then the sequence (\ref{ex;LS;Mx;eq}) is a Leonard system on $\Mx$.
\end{example}

Let $\Phi$ be a Leonard system in Definition \ref{Def;LS}. 
Each of the following is a Leonard system on $\sf V$:
\begin{gather*}
\Phi^* := ({\sf A}^*; {\sf A}; \{{\sf E}^*_i\}^{\sf d}_{i=0}; \{{\sf E}_i\}^{\sf d}_{i=0}), \\
\Phi^{\downarrow} := ({\sf A}; {\sf A}^*; \{{\sf E}_i\}^{\sf d}_{i=0}; \{{\sf E}^*_{{\sf d}-i}\}^{\sf d}_{i=0}),
\qquad 
\Phi^{\Downarrow} := ({\sf A}; {\sf A}^*; \{{\sf E}_{{\sf d}-i}\}^{\sf d}_{i=0}; \{{\sf E}^*_{i}\}^{\sf d}_{i=0}).
\end{gather*}
For $0 \leq i \leq {\sf d}$, let $\tht^*_i$ denote the eigenvalue of ${\sf A}^*$ associated with ${\sf E}^*_i$. By \cite[Theorem 3.2]{PT;2lin} there exist nonzero scalars $\{\varphi_i\}^{\sf d}_{i=0}$ and a $\C$-algebra homomorphism $\natural: \mcal{A} \to M_{{\sf d}+1}(\C)$ such that
\begin{equation*}
{\sf A}^{\natural} = 
\scalemath{0.8}{
\begin{bmatrix}
\tht_0 &&&&& {\bf 0} \\
1 & \tht_1 &\\
& 1 & \tht_2\\
&& \cdot & \cdot \\
&&& \cdot & \cdot \\
{\bf 0} &&&&1&\tht_{\sf d}
\end{bmatrix}},
\qquad \qquad
{\sf A}^{*\natural}=
\scalemath{0.8}{
\begin{bmatrix}
\tht^*_0 & \varphi_1 &&&&{\bf 0} \\
& \tht^*_1 & \varphi_2 &\\
&& \tht^*_2 & \cdot \\
&&&\cdot & \cdot & \\
&&&& \cdot & \varphi_{\sf d} \\
{\bf 0} &&&&& \tht^*_{\sf d}
\end{bmatrix}}.
\end{equation*}
We call the sequence $\{\varphi_i\}^{\sf d}_{i=1}$ the {\it first split sequence} of $\Phi$. 
We let $\{\phi_i\}^{\sf d}_{i=0}$ denote the first split sequence of 
$\Phi^{\Downarrow}$ and call this the {\it second split sequence} of $\Phi$.
By the {\it parameter array} of $\Phi$ we mean the sequence
\begin{equation*}
p(\Phi):=(\{\tht_i\}^{\sf d}_{i=0}, \{\tht^*_i\}^{\sf d}_{i=0}, \{\varphi_i\}^{\sf d}_{i=1}, \{\phi_i\}^{\sf d}_{i=1}).
\end{equation*}
Let $\Psi$ denote a Leonard system in a $\C$-algebra $\mcal{B}$. We say that $\Psi$ is {\it isomorphic} to $\Phi$ whenever there is a $\C$-algebra isomorphism $\alpha :\mcal{A} \to \mcal{B}$ such that $\Psi=\Phi^{\alpha} := ({\sf A}^{\alpha}; {\sf A}^{*\alpha}; \{{\sf E}^{\alpha}_i\}^{\sf d}_{i=0}; \{{\sf E}^{*\alpha}_i\}^{\sf d}_{i=0})$. In \cite[Theorem 1.9]{PT;2lin} Terwilliger classified Leonard systems by using parameter arrays, and characterized the set of parameter arrays of Leonard systems with diameter ${\sf d}$. Moreover, he displayed all the parameter arrays over $\C$ in \cite{PT:PA}. We recall the $q$-Racah family of parameter arrays, that is the most general family.

\begin{example}\cite[Example 5.3]{PT:PA}\label{ex;q-racah;PA} ($q$-Racah type)
For $0 \leq i \leq {\sf d}$ define
\begin{align}
\label{theta_i}
\theta_i 	& =  \theta_0+h(1-q^{i})(1-sq^{i+1})q^{-i}, \\
\label{theta^*_i}
\theta^*_i 	& =  \theta^*_0+h^*(1-q^{i})(1-s^*q^{i+1})q^{-i},
\end{align}
and for $1\leq i \leq {\sf d}$ define
\begin{align}
\label{varphi_i}
\varphi_i	& =  hh^{*}q^{1-2i}(1-q^{i})(1-q^{i-{\sf d}-1})(1-r_1q^i)(1-r_2q^i), \\
\label{phi_i}
\phi_i 		& =  hh^{*}q^{1-2i}(1-q^{i})(1-q^{i-{\sf d}-1})(r_1-s^*q^i)(r_2-s^*q^i)/s^*,
\end{align} 
where $\tht_0$ and $\tht_0^*$ are scalars in $\C$, and
where $h, h^*, s, s^*, r_1, r_2$ are nonzero scalars in $\C$ such that
$r_1r_2=ss^*q^{{\sf d}+1}$.
To avoid degenerate situations assume that 
\begin{itemize}
\item[(i)] none of $q^i, r_1q^i, r_2q^i, s^*q^i/r_1, s^*q^i/r_2$ is equal to 1 for $1\leq i \leq {\sf d}$,
\item[(ii)] neither of $sq^i, s^*q^i$ is equal to 1 for $2 \leq i \leq 2{\sf d}$.
\end{itemize}
Then the sequence
$(\{\tht_i\}^{\sf d}_{i=0}, \{\tht^*_i\}^{\sf d}_{i=0},\{\varphi_i\}^{\sf d}_{i=1}, \{\phi_i\}^{\sf d}_{i=1})$
is a parameter array over $\C$.
This parameter array is said to have {\it $q$-Racah type}. 
\end{example}
\noindent
We say that {\it $\Phi$ has $q$-Racah type} whenever its parameter array has $q$-Racah type.

\medskip
Let $u$ be a nonzero vector in ${\sf E}_0{\sf V}$. By \cite[Lemma 10.2]{PT:Madrid}, the sequence $\{{\sf E}^*_iu\}^{\sf d}_{i=0}$ is a basis for ${\sf V}$, called a {\it $\Phi$-standard basis} for {\sf V}. The following is a characterization of the $\Phi$-standard basis.
\begin{lemma}{\rm\cite[Lemma 10.4]{PT:Madrid}}\label{Phi-sb}
Let $\{v_i\}^{\sf d}_{i=0}$ denote a sequence of vectors in ${\sf V}$, not all $0$.
Then this sequence is a $\Phi$-standard basis for $\sf V$ if and only if both
{\rm(i)} $v_i \in {\sf E}^*_i {\sf V}$ for $0 \leq i \leq {\sf d};$ {\rm(ii)} $\sum^{\sf d}_{i=0} v_i \in {\sf E}_0 {\sf V}$.
\end{lemma}

Consider the Leonard system $\Phi$ from Definition \ref{Def;LS} and its corresponding parameter array $p(\Phi)$. The matrix representing ${\sf A}^*$ relative to a $\Phi$-standard basis is
\begin{equation*}
\diag(\tht^*_0, \tht^*_1, \tht^*_2, \ldots, \tht^*_{\sf d}).
\end{equation*}
Moreover, the matrix representing ${\sf A}$ relative to a $\Phi$-standard basis is the tridiagonal matrix
\begin{equation}\label{[A^(flat)]}
\scalemath{0.9}{
\begin{bmatrix}
a_0 & b_0 &&& {\bf 0}\\
c_1 & a_1 & b_1 && \\
& c_2 & a_2 & \ddots & \\
&& \ddots &\ddots & b_{{\sf d}-1} \\
{\bf 0}&&&c_{\sf d}&a_{\sf d}
\end{bmatrix}},
\end{equation}
where $\{a_i\}^{\sf d}_{i=0},\{b_i\}^{{\sf d}-1}_{i=0}, \{c_i\}^{\sf d}_{i=1}$ are some scalars in $\mathbb{C}$. We call $a_i, b_i, c_i$ the {\it intersection numbers of $\Phi$}. Note that the matrix (\ref{[A^(flat)]}) has constant row sum $\tht_0$ \cite[Lemma 10.5]{PT:Madrid}.

\begin{example}\label{Mx;S-basis;I-num}
Let $\Phi$ be the Leonard system in Example \ref{ex;LS;Mx}. Then $\{A_i\hat{x}\}^D_{i=0}$ form a $\Phi$-standard basis for $\Mx$. Recall the scalars $a_i,b_i,c_i$ from (\ref{intersection;numbers}). These are the intersection numbers of $\Phi$\cite[Theorem 4.1(vi)]{PT;T-alg;II}. 
\end{example}

\begin{note}\label{Note;S4}
Recall from Section \ref{S3;DRG} that $\Ga$ is a $Q$-polynomial distance-regular graph.
Let $\Phi=\Phi(\Ga)$ denote the Leonard system (\ref{ex;LS;Mx;eq}) associated with $\Ga$. We say that {\it $\Ga$ has $q$-Racah type} when $\Phi$ has $q$-Racah type. For the rest of the paper, assume that $\Ga$ has $q$-Racah type. Because $p(\Phi)$ has $q$-Racah type, it satisfies (\ref{theta_i})--(\ref{phi_i}) for some scalars $h,h^*,s,s^*,r_1,r_2$. We fix this notation for the rest of the paper. Referring to this notation, whenever we encounter square roots, these are interpreted as follows. We fix square roots $s^{1/2}, s^{*1/2}, r_1^{1/2},r_2^{1/2}$ such that $r_1^{1/2}r_2^{1/2}=s^{1/2}s^{*1/2}q^{(D+1)/2}$.
\end{note}

Let $\Phi$ denote a Leonard system in Definition {\rm \ref{Def;LS}}. Let $p(\Phi)=p(\Phi;q)$ denote the parameter array of $\Phi$ that has $q$-Racah type in Example \ref{ex;q-racah;PA}. In the following proposition we describe the parameter array that has $q^{-1}$-Racah type. 
\begin{proposition}\label{q-inv;Racah}
{\rm ($q^{-1}$-Racah type)}
For $0 \leq i \leq {\sf d}$ define
\begin{align}
\label{tht;q-inv}
& \theta'_i  =  \theta'_0+h'(1-q^{-i})(1-s'q^{-i-1})q^{i}, \\
\label{tht*;q-inv}
& {\theta^*_i}'  = {\theta^*_0}'+{h^*}'(1-q^{-i})(1-{s^*}'q^{-i-1})q^{i},
\end{align}
and for $1\leq i \leq {\sf d}$ define
\begin{align}
\label{varphi;q-inv}
& {\varphi_i}'	 =  h'{h^{*}}'q^{-1+2i}(1-q^{-i})(1-q^{-i+{\sf d}+1})(1-r'_1q^{-i})(1-r'_2q^{-i}), \\
\label{phi;q-inv}
& {\phi_i}'  =  h'{h^{*}}'q^{-1+2i}(1-q^{-i})(1-q^{-i+{\sf d}+1})(r'_1-{s^*}'q^{-i})(r'_2-{s^*}'q^{-i})/{s^*}',
\end{align} 
where 
\begin{align}
\label{rel;q-inv(0)}
& \tht'_0=\tht_0, && {\tht^*_0}'=\tht^*_0, &&\\
\label{rel;q-inv(1)}
& h' = hsq, && s'=s^{-1}, && r'_1 = r_1^{-1}, && \\
\label{rel;q-inv(2)}
& {h^*}' = h^*s^*q, && {s^*}'=s^{*-1}, && r'_2 = r_2^{-1}. && 
\end{align}
Then the sequence $p(\Phi;q^{-1}):=(\{\tht'_i\}^{\sf d}_{i=0}, \{{\tht^*_i}'\}^{\sf d}_{i=0}, \{\varphi'_i\}^{\sf d}_{i=1}, \{\phi'_i\}^{\sf d}_{i=1})$ is equal to $p(\Phi;q)$. Therefore $p(\Phi;q^{-1})$ is the parameter array that has $q^{-1}$-Racah type. 
\end{proposition}
\begin{proof}
Using (\ref{rel;q-inv(0)})--(\ref{rel;q-inv(2)}) one checks that $\tht'_i=\tht_i, {\tht^*_i}'=\tht^*_i$ for $0 \leq i \leq {\sf d}$ and $\varphi'_i=\varphi_i, \phi'_i=\phi_i$ for $1\leq i \leq {\sf d}$. It follows that $p(\Phi;q^{-1})$ is the parameter array of $\Phi$. By definition of $q$-Racah type in Example \ref{ex;q-racah;PA},  $p(\Phi;q^{-1})$ has $q^{-1}$-Racah type.
\end{proof}

\section{$T$-module $\W$}\label{T-mod;W}

We recall the $T$-module $\W$ from the last paragraph 
in Section \ref{S3;DRG}. Note that 
$\W$ is decomposed into the direct sum of two irreducible 
$T$-modules $\Mx$ and $\Mxp$ \cite[Section~5]{JHL}. 
We first discuss $\Mx$ and 
its associated polynomials. Recall from Example \ref{ex;LS;Mx} 
that $\Phi := (A, A^*, \{E_i\}^D_{i=0}, \{E^*_i\}^D_{i=0})$ 
is a Leonard system on $\Mx$. Also recall from Example 
\ref{Mx;S-basis;I-num} that $\{A_i\hat{x}\}^D_{i=0}$ is the 
$\Phi$-standard basis for $\Mx$ and the intersection numbers 
$a_i,b_i,c_i$ of $\Phi$. Abbreviate $v_i = A_i\hat{x}$ for 
$0 \leq i \leq D$. Observe that $v_0 = \hat{C}^-_0=\hat{x}, 
v_i = \hat{C}^+_{i-1}+\hat{C}^-_i (1 \leq i \leq D-1),$ and 
$v_D = \hat{C}^+_{D-1}$.
We now define a sequence of polynomials 
$f_0, f_1, \ldots , f_D$ by $f_0 :=1$ and 
$$
xf_i = b_{i-1}f_{i-1}+a_if_i+c_{i+1}f_{i+1} \qquad (0 \leq i \leq D-1),
$$
where $f_{-1}=0$. Then by \cite[Theorem 13.4]{PT:Madrid} 
we have
\begin{equation}\label{shift;fi}
f_i(A)v_0 = v_i \qquad \qquad  (0 \leq i \leq D).
\end{equation}
For $0 \leq i \leq D$, define the scalars $k_i$ by
\begin{equation}\label{scalar;ki}
k_i = b_0b_1\cdots b_{i-1}/c_1c_2 \cdots c_i.
\end{equation}
With the scalars $k_i$ and the polynomials $f_i$ 
we define the polynomial $F_i $ by
\begin{equation}\label{def;F}
F_i = f_i/k_i \qquad \qquad (0 \leq i \leq D).
\end{equation}
One routinely checks that 
$$
xF_i = b_iF_{i+1} + a_iF_i + c_iF_{i-1} \qquad \qquad (0 \leq i \leq D-1),
$$
where $F_{-1}=0$. By \cite[Theorem 23.2]{PT:Madrid}, it follows that for $0 \leq i \leq D$
\begin{equation}\label{poly_Fi}
F_i(x) = \sum^i_{j=0} \frac{(\tht^*_i-\tht^*_0)(\tht^*_i-\tht^*_1)\cdots(\tht^*_i-\tht^*_{j-1})}{\varphi_1\varphi_2\cdots\varphi_j}(x-\tht_0)(x-\tht_1)\cdots(x-\tht_{j-1}).
\end{equation}
\begin{definition}\label{abcd}
With reference to the parameters $s,s^*,r_1,r_2,D$ associated with $p(\Phi)=p(\Phi;q)$, define the scalars $a,b,c,d$ by
\begin{equation}\label{rels;abcd}
a= \left(\frac{r_1r_2}{s^*q^D}\right)^{1/2}, \quad
b=\left( \frac{s^*}{r_1r_2q^D} \right)^{1/2}, \quad
c=\left( \frac{r_2s^*q^{D+2}}{r_1} \right)^{1/2}, \quad
d=\left( \frac{r_1s^*q^{D+2}}{r_2} \right)^{1/2}.
\end{equation}
We say that the scalars {\it $a,b,c,d$ are associated with $p(\Phi)$.}
\end{definition}
\noindent
Referring to Definition \ref{abcd}, we have the following equations which are useful for our calculation.
\begin{equation}\label{abcd;simple;formula}
\begin{array}{lllllllll}
&ab = q^{-D}, & \qquad & ac = r_2q, &\qquad &  ad = r_1q, &\\
&bc =s^*q/r_1, & \qquad &  bd = s^*q/r_2, & \qquad & cd = s^*q^{D+2}, &\qquad & abcd =s^*q^2.
\end{array}
\end{equation}

\begin{lemma}\label{conditions;abcd}
Let the scalars $a,b,c,d$ be as in Definition {\rm \ref{abcd}}. Then the following hold:
\begin{itemize}
\item[\rm (i)] none of $abq^i,acq^i,adq^i,bcq^i,bdq^i$ is equal to $1$
for $0 \leq i \leq D-1$,
\item[\rm(ii)] $cdq^i\ne1$ for $-D \leq i \leq D-2$, 
\item[\rm(iii)] $abcdq^i\ne 1$  for $0 \leq i \leq 2D-2$.
\end{itemize}
\end{lemma}
\begin{proof}
Use assumption (i), (ii) in Example \ref{ex;q-racah;PA} and (\ref{abcd;simple;formula}).
\end{proof}
Consider a finite sequence of the polynomials $\{p_i(y+y^{-1})\}^D_{i=0}$ which are defined by the scalars $a,b,c,d$ associated with $p(\Phi)$:
\begin{equation}\label{AWpoly;ab;finite}
p_i(y+y^{-1}) = 
p_i(y+y^{-1};a,b,c,d \mid q) 
 = {_4}\phi_3
\left(
\begin{matrix}
q^{-i}, ~ abcdq^{i-1}, ~ ay, ~  ay^{-1} \\
ab, ~ ac, ~ ad
\end{matrix}
~\middle| ~ q,~q
\right),
\end{equation}
where $y$ is indeterminate; cf. (\ref{AW;abcd}).
Applying (\ref{rels;abcd}) and the equation $r_1r_2=ss^*q^{D+1}$ to (\ref{AWpoly;ab;finite}) gives
\begin{equation}
\label{AWpoly;s,s*,r1,r2}
p_i(y+y^{-1}) = {_4}\phi_3
\left(
\begin{matrix}
q^{-i}, ~ s^*q^{i+1}, ~ (sq)^{1/2}y, ~  (sq)^{1/2}y^{-1} \\
q^{-D}, ~ r_1q, ~ r_2q
\end{matrix}
~\middle| ~ q,~q
\right). 
\end{equation}
\begin{lemma}\label{rel;F,p} 
Recall the polynomial sequences $\{F_i\}^D_{i=0}$ from {\rm(\ref{poly_Fi})} 
and $\{p_i(y+y^{-1})\}^D_{i=0}$ from {\rm(\ref{AWpoly;s,s*,r1,r2})}. 
Let $x$ be of the form 
\begin{equation}\label{x_form}
h(sq)^{1/2}(y+y^{-1}) + (\tht_0 - h - hsq),
\end{equation}
where $h,s,\tht_0$ are associated with $p(\Phi)$. Then 
\begin{equation}\label{F=p}
F_i(x) = p_i(y+y^{-1}), \qquad i=0,1,2,\ldots D.
\end{equation}
\end{lemma}
\begin{proof}
First we compute the left-hand side in (\ref{F=p}).
Put (\ref{x_form}) for $x$ in $(\ref{poly_Fi})$ and evaluate the result to obtain
$$
\sum^i_{j=0} \frac{(q^{-i};q)_j(s^*q^{i+1};q)_j(s^{1/2}q^{1/2}y;q)_j(s^{1/2}q^{1/2}y^{-1};q)_j}{(r_1q;q)_j(r_2q;q)_j(q^{-D};q)_j(q;q)_n} q^j.
$$
This is equal to the right-hand side of (\ref{AWpoly;s,s*,r1,r2}) by the definition of basic hypergeometric series. Therefore the result follows.
\end{proof}
\begin{remark}\label{q-Racah;poly} With the above discussion, 
pick an integer $j~(0 \leq i \leq D)$. Evaluating (\ref{AWpoly;s,s*,r1,r2})
at $y=s^{1/2}q^{1/2+j}$ (or $y=s^{-1/2}q^{-1/2-j}$) we get
$$
{_4}\phi_3
\left(
\begin{matrix}
q^{-i}, ~ s^*q^{i+1}, ~ q^{-j}, ~  sq^{j+1} \\
q^{-D}, ~ r_1q, ~ r_2q
\end{matrix}
~\middle| ~ q,~q
\right).
$$
By this and definition of the $q$-Racah polynomials \cite{AW2},
$\{p_i(y+y^{-1})\}^{D-1}_{i=0}$ are the $q$-Racah polynomials.
\end{remark}

We now consider $\Mxp$, the orthogonal complement of $\Mx$ in $\W$. From \cite[Section~6]{JHL} the sequence $\Phi^\p_i := (A,A^*,\{E^\p_i\}^{D-2}_{i=0}, \{E^{*\p}_i\}^{D-2}_{i=0})$ acts as a Leonard system on $\Mxp$, where $E^\p_i=E_{i+1}$ and $E^{*\p}_i=E^*_{i+1}$ for $0 \leq i \leq D-2$.
Define the vectors $\{v^\p_i\}^{D-2}_{i=0}$ by
\begin{equation}\label{basis;v_perp}
v^\p_i= \xi_{i+1}\hat{C}^+_i + \xi_{i+1}\epsilon_{i+1}\hat{C}^-_{i+1},
\end{equation}
where for $1\leq i \leq D-1$
\begin{equation}\label{xi;epsilon}
\xi_i = q^{1-i}(1-q^{i-D})(1-s^*q^{i+1}), \qquad \epsilon_i=\frac{(1-q^i)(1-s^*q^{D+i+1})}{q^D(1-q^{i-D})(1-s^*q^{i+1})}.
\end{equation}
Then the sequence $\{v^\p_i\}^{D-2}_{i=0}$ is a $\Phi^\p$-standard basis for $\Mxp$ \cite[Lemma 6.6]{JHL}.
Let $a^\p_i, b^\p_i, c^\p_i$ denote the intersection numbers of $\Phi^\p$ \cite[(83)]{JHL}. 
We define a sequence of polynomials 
$f^\p_0, f^\p_1, \ldots, f^\p_{D-2}$ by $f^\p_0 :=1$ and
$$
xf^\p_i = b^\p_{i-1}f^\p_{i-1}+a^\p_if^\p_i+c^\p_{i+1}f^\p_{i+1} \qquad (0 \leq i \leq D-3),
$$
where $f_{-1}=0$. By \cite[Theorem 13.4]{PT:Madrid} we have
\begin{equation}\label{shift;fi;perp}
f^\p_i(A)v^\p_0=v^\p_i \qquad \qquad (0 \leq i \leq D-2).
\end{equation}
For $0\leq i \leq D-2$ define the scalars $k^\p_i$ by 
\begin{equation}\label{scalar;kiperp}
k^\p_i = b^\p_0b^\p_1 \cdots b^\p_{i-1}/c^\p_1c^\p_2\cdots c^\p_i.
\end{equation}
With the scalars $k^\p_i$ and the polynomials $f^\p_i$ 
we define the polynomial $F^\p_i$ by
\begin{equation}\label{def;F;perp}
F^\p_i=f^\p_i/k_i^\p \qquad \qquad (0 \leq i \leq D-2).
\end{equation}
One routinely checks that 
$$
xF^\p_i = b^\p_iF^\p_{i+1} + a^\p_iF^\p_i + c^\p_iF^\p_{i-1} \qquad \qquad (0 \leq i \leq D-3),
$$
where $F^\p_{-1}=0$. Consider the parameter array $p(\Phi^\p)=(\{\tht^\p_i\}^{D-2}_{i=0}, \{\tht^{*\p}_i\}^{D-2}_{i=0}, \{\varphi^\p_i\}^{D-2}_{i=0}, \{\phi^\p_i\}^{D-2}_{i=0})$. Applying \cite[Theorem 23.2]{PT:Madrid} to $\Phi^\p$ we find that for  $0 \leq i \leq D-2$
\begin{equation}\label{poly_Fi;perp}
F^\p_i(x) = \sum^i_{j=0} \frac{(\tht^{*\p}_i-\tht^{*\p}_0)(\tht^{*\p}_i-\tht^{*\p}_1)\cdots(\tht^{*\p}_i-\tht^{*\p}_{j-1})}{\varphi^\p_1\varphi^\p_2\cdots\varphi^\p_j}(x-\tht^\p_0)(x-\tht^\p_1)\cdots(x-\tht^\p_{j-1}).
\end{equation}

\begin{lemma}{\rm\cite[Theorem 6.10, Corollary 6.11]{JHL}}\label{rels;p;p-perp}
The parameter array $p(\Phi^\p)$ has $q$-Racah type.
With reference to the scalars $h,h^*,s,s^*,r_1,r_2$ associated with $p(\Phi)$ and the scalars $h^\p, h^{*\p}, s^\p,s^{*\p},r_1^\p,r_2^\p$ associated with $p(\Phi^\p)$, their relations are as follows.
\begin{align}
\label{perp;s,r,h(1)}
&& & h^\p = hq^{-1}, && s^\p = sq^2, && r_1^\p = r_1q, && \\
\label{perp;s,r,h(2)}
&& & h^{*\p} = h^*q^{-1}, && s^{*\p} = s^*q^2, && r_2^\p = r_2q. && 
\end{align}
\end{lemma}

\noindent
Consider the scalars $a^\p, b^\p, c^\p, d^\p$ associated with $p(\Phi^\p)$ in a view of Definition \ref{abcd}. Then by Lemma \ref{rels;p;p-perp}, 
\begin{equation}\label{rels;abcd;perp}
a^\p= \left(\tfrac{r_1r_2}{s^*q^{D-2}}\right)^{1/2}, \quad
b^\p=\left( \tfrac{s^*}{r_1r_2q^{D-2}} \right)^{1/2}, \quad
c^\p=\left( \tfrac{r_2s^*q^{D+2}}{r_1} \right)^{1/2}, \quad
d^\p=\left( \tfrac{r_1s^*q^{D+2}}{r_2} \right)^{1/2}.
\end{equation}
Using Definition \ref{abcd} one can readily check that
\begin{equation*}
a^\p= aq, \qquad
b^\p=bq, \qquad
c^\p=c, \qquad
d^\p=d.
\end{equation*}
We define a finite sequence of the polynomials $\{p^\p_i(y+y)\}^{D-2}_{i=0}$ which are defined by the scalars $a^\p, b^\p,c^\p,d^\p$ associated with $p(\Phi^\p)$:
\begin{equation}\label{AWpoly;perp;finite}
p^\p_i(y+y^{-1}) := p_i(y+y^{-1}; a^\p, b^\p, c^\p, d^\p \mid q )=p_i(y+y^{-1}; aq, bq, c, d \mid q ).
\end{equation}
Using (\ref{rels;abcd;perp}) we find
\begin{equation}\label{AWpoly;s,s*;perp}
p^\p_i(y+y^{-1}) = 
{_4}\phi_3
\left(
\begin{matrix}
q^{-i}, ~ s^*q^{i+3}, ~ (sq)^{1/2}qy, ~  (sq)^{1/2}qy^{-1} \\
q^{-D+2}, ~ r_1q^2, ~ r_2q^2
\end{matrix}
~\middle| ~ q,~q
\right).
\end{equation}

\begin{lemma}\label{rel;F,p;perp} 
Recall the polynomial sequences $\{F^\p_i\}^{D-2}_{i=0}$ from {\rm(\ref{poly_Fi;perp})} 
and $\{p^\p_i(y+y^{-1})\}^{D-2}_{i=0}$ from {\rm(\ref{AWpoly;s,s*;perp})}. 
Let $x$ be of the form 
\begin{equation*}
h(sq)^{1/2}(y+y^{-1}) + (\tht_0 - h - hsq),
\end{equation*}
where $h,s,\tht_0$ are associated with $p(\Phi)$. Then 
\begin{equation*}
F^\p_i(x) = p^\p_i(y+y^{-1}), \qquad i=0,1,2,\ldots D-2.
\end{equation*}
\end{lemma}
\begin{proof}
Similar to Lemma \ref{rel;F,p}.
\end{proof}

Recall the $\Phi$-standard basis $\{v_i\}^D_{i=0}$ for $\Mx$ and the $\Phi^\p$-standard basis $\{v_i^\p\}^{D-2}_{i=0}$ for $\Mxp$. By \cite[Lemma 8.3]{JHL}, for $1 \leq i \leq D-1$
\begin{equation}\label{lem8.3}
\hat{C}^+_{i-1} = \frac{\epsilon_i}{\epsilon_i-1}v_i + \frac{1}{\xi_i(1-\epsilon_i)}v^\p_{i-1},
\qquad 
\hat{C}^-_{i} = \frac{1}{1-\epsilon_i}v_i + \frac{1}{\xi_i(\epsilon_i-1)}v^\p_{i-1}.
\end{equation}

\begin{lemma} \label{C+-;f;fp}
For $1 \leq i \leq D-1$,
\begin{align}
\label{C+;f;fp}
& \hat{C}^+_{i-1} = \frac{\epsilon_i}{\epsilon_i-1}f_i(A)v_0 + \frac{1}{\xi_i(1-\epsilon_i)}f^\p_{i-1}(A)v_0^\p, \\
\label{C-;f;fp}
& \hat{C}^-_{i} = \frac{1}{1-\epsilon_i}f_i(A)v_0 + \frac{1}{\xi_i(\epsilon_i-1)}f^\p_{i-1}(A)v^\p_0.
\end{align}
And
$$
\hat{C}^-_0 = f_0(A)v_0, \qquad \qquad \hat{C}^+_{D-1} = f_D(A)v_0.
$$
\end{lemma}
\begin{proof}
To get (\ref{C+;f;fp}), (\ref{C-;f;fp}) use (\ref{shift;fi}), (\ref{shift;fi;perp}) together
with (\ref{lem8.3}).
Note that $\hat{C}^-_0 = v_0$ and $\hat{C}^+_{D-1} = v_D.$ The result follows.
\end{proof}

We finish this section with a few comments. 
The following lemmas will be useful in the sequel.

\begin{lemma}\label{norm;C;pm} 
Recall the Hermitian inner product $\< \cdot, \cdot \>_V$ from the first
paragraph in Section {\rm \ref{S3;DRG}}.
For $0 \leq i \leq D-1$,
\begin{align*}
\Vert \hat{C}^-_i \Vert^2 & = \left(\frac{s^*q^D}{r_1r_2}\right)^i\frac{(q^{1-D}, s^*q^2, r_1q, r_2q;q)_i}{(q,s^*q^{D+2},s^*q/r_1,s^*q/r_2;q)_i}, \\
\Vert \hat{C}^+_i \Vert^2 & =  -\frac{s^*(1-r_1q)(1-r_2q)}{(r_1-s^*q)(r_2-s^*q)}\left(\frac{s^*q^D}{r_1r_2}\right)^i\frac{(q^{1-D}, s^*q^2, r_1q^2, r_2q^2; q)_i}{(q, s^*q^{D+2}, s^*q^2/r_1, s^*q^2/r_2; q)_i}.
\end{align*}
\end{lemma}
\begin{proof} Evaluate \cite[Lemma 5.7]{JHL} using \cite[(17)--(20)]{JHL} and \cite[Corollary 4.9]{JHL}. The results routinely follow.
\end{proof}
\noindent
Note that $\Vert \hat{C}^-_i \Vert^2, \Vert \hat{C}^+_i \Vert^2$ are all positive integral,
since they are equal to the cardinality of $C^-_i, C^+_i$, respectively.
\begin{lemma}{\rm\cite[Section 18]{PT:LP_qPoly}}\label{k;kp;nu;nup}
\begin{itemize}
\item[\rm (i)] Recall the scalar $k_i$ $(0 \leq i \leq D)$ from {\rm(\ref{scalar;ki})}. Then
$$
k_i = \frac{(r_1q;q)_i(r_2q;q)_i(q^{-D};q)_i(s^*q;q)_i(1-s^*q^{2i+1})}{s^iq^i(q;q)_i(s^*q/r_1;q)_i(s^*q/r_2;q)_i(s^*q^{D+2};q)_i(1-s^*q)}.
$$
To get $k^*_i$, replace $(s,s^*)$ with $(s^*,s)$.
\item[\rm (ii)]
Recall the scalar $k_i^\p~(0 \leq i \leq D-2)$ from {\rm (\ref{scalar;kiperp})}. Then
$$
k^\p_i = \frac{(r_1q^2;q)_i(r_2q^2;q)_i(q^{-D+2};q)_i(s^*q^3;q)_i(1-s^*q^{2i+3})}{s^iq^{3i}(q;q)_i(s^*q^2/r_1;q)_i(s^*q^2/r_2;q)_i(s^*q^{D+2};q)_i(1-s^*q^3)}.
$$
To get $k^{*\p}_i$, replace $(s,s^*)$ with $(s^*,s)$.
\item[\rm (iii)] Let $\nu$ be the scalar such that $tr(E_0E^*_0)=\nu^{-1}$. Then
$$
\nu = \frac{(sq^2;q)_D(s^*q^2;q)_D}{r_1^Dq^D(sq/r_1;q)_D(s^*q/r_1)_D}.
$$
\item[\rm (iv)] Let $\nu^\p$ be the scalar such that $tr(E^\p_0E^{*\p}_0)=\nu^{\p-1}$. Then
$$
\nu^\p = \frac{(sq^4;q)_{D-2}(s^*q^4;q)_{D-2}}{r^{D-2}_1q^{2D-4}(sq^2/r_1;q)_{D-2}(s^*q^2/r_1)_{D-2}}.
$$
\end{itemize}
\begin{proof}
(i), (iii): From \cite[p. 37]{PT:LP_qPoly}.\\
(ii): In part (i) replace $D$ by $D-2$ and use (\ref{perp;s,r,h(1)}), (\ref{perp;s,r,h(2)}) to get the result.\\
(iv): In part (iii) replace $D$ by $D-2$ and use (\ref{perp;s,r,h(1)}), (\ref{perp;s,r,h(2)}) to get the result.
\end{proof}
\end{lemma}

\section{The universal DAHA of type $(C^{\vee}_1, C_1)$}\label{Section;UDAHA}

In this section we discuss the universal DAHA of type $(C^{\vee}_1,C_1)$ 
and its properties. For notational convenience define an index set $\I=\{0,1,2,3\}$.
\begin{definition}\label{Def;UDAHA}\cite[Definition 3.1]{PT:DAHA} 
The {\it universal DAHA} of type $(C^{\vee}_1, C_1)$ is the $\C$-algebra $\H$ defined by generators $\{t^{\pm1}_n\}_{n \in \I}$ and relations
\begin{equation*}
{\rm (i)}~ t_nt_n^{-1} = t_n^{-1}t_n=1 ~ (n \in \I); \quad
{\rm (ii)}~ t_n+t_n^{-1} \text{ is central } ~ (n \in \I); \quad
{\rm (iii)}~ t_0t_1t_2t_3 = q^{-1/2}.
\end{equation*}
\end{definition}
\noindent
For notational convenience define the following elements in $\H$:
\begin{align*}
& \Y=t_0t_1, && \X=t_3t_0, && \wt\X=t_1t_2=q^{-1/2}(t_3t_0)^{-1}, &&\\
& \A=\Y+\Y^{-1}, && \B = \X+\X^{-1}, && \wt\B=\wt\X+\wt\X^{-1}.
\end{align*}

\begin{lemma}\label{antiauto;Hq}
There exists a unique antiautomorphism $\dagger$ of $\H$ that sends
$$
t_0 \mapsto t_1, \qquad 
t_1 \mapsto t_0, \qquad
t_2 \mapsto t_3, \qquad
t_3 \mapsto t_2.
$$
Moreover $\dagger^2=1$.
\end{lemma}
\begin{proof} Use Definition \ref{Def;UDAHA}.
\end{proof}
\noindent
By Lemma \ref{antiauto;Hq} we have the following:
\begin{equation}\label{X,Y;dagger}
\Y^{\dagger}=\Y, \qquad \X^{\dagger}=\wt{\X}, \qquad \wt{\X}^{\dagger}=\X.
\end{equation}

\noindent
In \cite[Section~11]{JHL} the author showed that the space $\W$ is an $\H$-module as well as a $\T$-module, and displayed its module structure in detail. In the present paper, for the purpose of our study we will twist $\W$ via a certain $\C$-algebra automorphism of $\H$. Recall the $\H$-module $\W$ from \cite[Section~11]{JHL}. Consider a $\C$-algebra automorphism $\rho: \H \to \H$ that sends
\begin{align*}
&& t_0 \mapsto t_1, &&
t_1 \mapsto t_0, &&
t_2 \mapsto t_0^{-1}t_3t_0, &&
t_3 \mapsto t_1t_2t_1^{-1}. &&
\end{align*}
Observe that ${\rho}^2=1$. There exists an $\H$-module structure on $\W$, called  {\it $\W$  twisted via $\rho$}, that behaves as follows; for all $h \in \H, w \in \W$, the vector $h.w$ computed in $\W$ twisted via $\rho$ coincides with the vector $h^{\rho}.w$ computed in the original $\H$-module $\W$. 
We display the $\H$-module structure $\W$ twisted via $\rho$ in Appendix \ref{Apdx;H-mod} in detail. 
For the rest of the paper, we regard an $\H$-module $\W$ as the $\H$-module $\W$ twisted via $\rho$. The $\H$-module structure on $\W$ is determined by the scalars $q,s,s^*,r_1,r_2,D$ associated with $p(\Phi)$. We denote this module by $\W(s,s^*,r_1,r_2,D;q)$.

\smallskip
\noindent
Define the scalars $\{\k_n\}_{n\in \I}$ by
\begin{equation}\label{k_n}
\k_0 = \left(\frac{1}{q^{D}}\right)^{1/2}, \qquad
\k_1 = \left(\frac{r_1r_2}{s^*}\right)^{1/2}, \qquad
\k_2 = \left( \frac{r_2}{r_1} \right)^{1/2}, \qquad
\k_3 = \left(s^*q^{D+1}\right)^{1/2}.
\end{equation}

\begin{lemma}\label{kn;not=1}
For each $n \in \I$, the scalar $\kappa_n$ is not equal to $\pm 1$.
\end{lemma}
\begin{proof}
Use assumption (i), (ii) in Example \ref{ex;q-racah;PA}.
\end{proof}

\begin{lemma}\label{kn;diagonalizable}
For each $n \in \I$, $t_n$ is diagonalizable on $\W(s,s^*,r_1,r_2,D;q)$.
\end{lemma}
\begin{proof}
By \cite[Lemma 11.9]{JHL} for each $n\in \I$ it follows $(t_n+t_n^{-1}).w = (\k_n+\k_n^{-1}).w$ for all $w \in \W$. So the minimal polynomial of $t_n$ is $(x-\k_n)(x-\k_n^{-1})$, and this has distinct roots by Lemma \ref{kn;not=1}. The result follows.
\end{proof}
\noindent
On $\W(s,s^*,r_1,r_2,D;q)$ the action of $\X$ on $\{\hat{C}^{\pm}_i\}^{D-1}_{i=0}$
as follows\cite[Lemma 11.8]{JHL}:
\begin{equation}\label{X-action;W}
\X.\hat{C}^-_i = q^i(s^*q)^{1/2}\hat{C}^-_i, \qquad 
\X.\hat{C}^+_i = q^{-i-1}(s^*q)^{-1/2}\hat{C}^+_i \qquad (0 \leq i \leq D-1).
\end{equation}
The action of $\Y$ on $\{\hat{C}^{\pm}_i\}^{D-1}_{i=0}$ 
is given as a linear combination of four terms; see Appendix \ref{Action;Y}.
Recall the generators $A, A^*, \wt{A}^*$ of the algebra $\T$ 
and the elements $\A, \B, \wt{\B}$ of $\H$. 
The following theorem explains how the $\T$-action on $\W$ is related to the $\H$-action.
\begin{theorem}{\rm \cite[Theorem 12.1]{JHL}}\label{MainThm;JHL} On $\W$,
\begin{enumerate}
\item[\rm (i)] $A$ acts as $h(sq)^{1/2}\A+(\tht_0-h-hsq)$;
\item[\rm (ii)] $A^*$ acts as $h^*(s^*q)^{1/2}\B+(\tht_0^*-h^*-h^*s^*q)$;
\item[\rm (iii)] $\wt A^*$ acts as $\wt h^*(\wt s^*q)^{1/2}\wt\B + (\wt\tht_0^*-\wt h^*-\wt h^*\wt s^*q)$;
\item[\rm(iv)] $(t_0-\k^{-1}_0)(\k_0-\k_0^{-1})^{-1}$ acts as the projection of $\W$ onto $\Mx$;
\item[\rm(v)] $(t_1-\k^{-1}_1)(\k_1-\k_1^{-1})^{-1}$ acts as the projection of $\W$ onto $\MC$.
\end{enumerate}
\end{theorem}

\medskip
We now consider the scalars $a,b,c,d$ from Definition \ref{abcd}.
Using the relations (\ref{rels;abcd}) we can describe the $\H$-module
$\W(s,s^*,r_1,r_2,D;q)$ in terms of $a,b,c,d$ and $q$ as follows.
\begin{lemma}\label{Hq-mod;abcd}
Let the scalars $a,b,c,d$ be as in Definition \ref{abcd}.
Consider the block diagonal matrices ${\bf T}_n (n \in \I):$
\begin{align*}
&{\bf T}_0  = \bdiag\Bigl[ \tau_0(0), \tau_0(1), \ldots, \tau_0(D-1), [(ab)^{1/2}] \Bigr], \\
&{\bf T}_1  = \bdiag\Bigl[ \tau_1(0), \tau_1(1), \ldots, \tau_1(D-1)\Bigr], \\
&{\bf T}_2  = \bdiag\Bigl[ \tau_2(0), \tau_2(1), \ldots, \tau_2(D-1)\Bigr], \\
&{\bf T}_3  = \bdiag\Bigl[ \tau_3(0), \tau_3(1), \ldots, \tau_3(D-1), [(cdq^{-1})^{-1/2}]\Bigr],
\end{align*}
where $\tau_0(0) = \left[ (ab)^{1/2}\right]$ and for $1 \leq i \leq D-1$,
$$
\tau_0(i) =
(ab)^{-1/2}
\begin{bmatrix}
\frac{(1-abq^i)(1-abcdq^{i-1})}{1-abcdq^{2i-1}} +ab& &
-\frac{(1-abq^i)(1-abcdq^{i-1})}{1-abcdq^{2i-1}} \vspace{0.3cm}\\
\frac{ab(1-q^i)(1-cdq^{i-1})}{1-abcdq^{2i-1}} & &
-\frac{ab(1-q^i)(1-cdq^{i-1})}{1-abcdq^{2i-1}} + ab
\end{bmatrix},
$$
and for $0 \leq i \leq D-1$,
$$
\tau_1(i) =
({ab^{-1}})^{1/2}
\begin{bmatrix}
\frac{(1-bcq^i)(1-bdq^i)}{1-abcdq^{2i}} + \frac{b}{a} &&
-\frac{b}{a}\frac{(1-adq^i)(1-acq^i)}{1-abcdq^{2i}}  \vspace{0.3cm}\\
\frac{(1-bcq^i)(1-bdq^i)}{1-abcdq^{2i}} &&
\frac{b}{a}\left( 1-\frac{(1-adq^i)(1-acq^i)}{1-abcdq^{2i}} \right)
\end{bmatrix},
$$
and for $0 \leq i \leq D-1$,
$$
\tau_2(i)=
(cd)^{-1/2}
\begin{bmatrix}
\frac{1}{aq^i}\left( 1-\frac{(1-adq^i)(1-acq^i)}{1-abcdq^{2i}} \right) & &
\frac{bcdq^i (1-adq^i)(1-acq^i)}{1-abcdq^{2i}} \vspace{0.3cm}\\
-\frac{1}{bq^i}\frac{(1-bcq^i)(1-bdq^i)}{1-abcdq^{2i}} & &
acdq^i\left(\frac{(1-bcq^i)(1-bdq^i)}{1-abcdq^{2i}}+\frac{b}{a}\right)
\end{bmatrix},
$$
and $\tau_3(0) = \left[ (cdq^{-1})^{1/2} \right]$
and for $1\leq i \leq D-1$,
$$
\tau_3(i)=
(cdq^{-1})^{-1/2}
\begin{bmatrix}
\frac{1}{q^i}\left( 1-\frac{(1-q^i)(1-cdq^{i-1})}{1-abcdq^{2i-1}} \right) & & 
\frac{1}{abq^i} \frac{(1-abq^i)(1-abcdq^{i-1})}{1-abcdq^{2i-1}} \vspace{0.3cm}\\
-\frac{abcdq^{i-1}(1-q^i)(1-cdq^{i-1})}{1-abcdq^{2i-1}} & &
cdq^{i-1}\left( \frac{(1-abq^i)(1-abcdq^{i-1})}{1-abcdq^{2i-1}} + ab\right)
\end{bmatrix}.
$$ 
Then there exists an $\H$-module structure on $\W$ such that
for $n \in \I$ the matrix $\T_n$ represents the generator $t_n$ relative to $\{\hat{C}^{\pm}_i\}^{D-1}_{i=0}$. 
\end{lemma}
\begin{proof}
In Definition \ref{H-mod;s,r,h}, replace $s,s^*,r_1,r_2,D$ by $a,b,c,d$ using relations (\ref{rels;abcd}). 
\end{proof}
\noindent
For the rest of the paper we let $\W(a,b,c,d;q)$ denote the $\H$-module in Lemma \ref{Hq-mod;abcd}.

\begin{lemma}\label{X-action;abcd}
Referring to the $\H$-module $\W(a,b,c,d;q)$ in Lemma {\rm \ref{Hq-mod;abcd}}, the action of $\X$ on $\{\hat{C}^{\pm}_i\}^{D-1}_{i=0}$ is as follows.
\begin{equation*}
\begin{cases}
\X.\hat{C}^-_i = q^{i-\frac{1}{2}}(abcd)^{1/2}\hat{C}^-_i 
& \text{for} \quad i=0,1,\ldots, D-1, \\
\X.\hat{C}^+_{i-1} = q^{-i+\frac{1}{2}}(abcd)^{-1/2}\hat{C}^+_{i-1} 
& \text{for} \quad  i=1,2,\ldots, D.
\end{cases}
\end{equation*}
\end{lemma}
\begin{proof}
By Definition \ref{abcd}, $(s^*q)^{1/2}=q^{-1/2}(abcd)^{1/2}$. By this and (\ref{X-action;W}) the result follows.
\end{proof}

\section{Nonsymmetric Laurent polynomials $\varepsilon^{\pm}_i$}

In this section we define certain nonsymmetric Laurent polynomials $\{\varepsilon^{\pm}_i\}^{D-1}_{i=0}$ and discuss their properties.
Let $\C[y,y^{-1}]$ denote the space of Laurent polynomials with a variable $y$.
Recall the $\Phi$-standard basis $\{v_i\}^D_{i=0}$ for $\Mx$ and 
the $\Phi^\p$-standard basis $\{v^\p_i\}^{D-2}_{i=0}$ for $\Mxp$. 
Define the Laurent polynomial
\begin{equation}\label{poly;g}
g[y] := \left(\tfrac{s^*}{r_1r_2q^D}\right)^{1/2}\tfrac{(1-s^*q^2)(1-s^*q^3)}{(1-s^*q/r_1)(1-s^*q/r_2)}\left( y - \left(\tfrac{r_1r_2}{s^*q^D}\right)^{1/2} - \left(\tfrac{s^*}{r_1r_2q^D} \right)^{1/2} + \tfrac{1}{q^D}y^{-1}\right).
\end{equation}
\begin{lemma}\label{g(v0)=v0p}
Recall $\Y=t_0t_1$. On the $\H$-module $\W(s,s^*,r_1,r_2,D;q)$,
$$g[\Y]v_0 = v_0^\p.$$
\end{lemma}

\begin{proof}
Abbreviate $u=v_0^\p-g[\Y]v_0$.
We show that $u=0$. 
By (\ref{basis;v_perp}), $v_0^\p=\xi_1\hat{C}^+_0+\xi_1\epsilon_1\hat{C}^-_1$.
Next evaluate $g[\Y]v_0$ using Lemma \ref{Y-action}(a), Lemma \ref{Yinv-action}(a) and $v_0=\hat{C}^-_0$.
Using these comments, we evaluate $u$ to get zero. The result follows.
\end{proof}

\noindent
We now define nonsymmetric Laurent polynomials $\varepsilon^{-}_i,\varepsilon^+_i (0 \leq i\leq D-1)$. 
Recall the equation (\ref{C+;f;fp}), that is,
$$
\hat{C}^+_{i-1}  =\frac{\epsilon_i}{\epsilon_i-1}f_i(A)v_0 + \frac{1}{\xi_i(1-\epsilon_i)}f^\p_{i-1}(A)\v0.
$$
Applying (\ref{def;F}) and (\ref{def;F;perp}) to the right-hand side of the above equation gives
\begin{equation}\label{RHS;C+i-1}
\frac{\epsilon_i}{\epsilon_i-1}k_iF_i(A)v_0 + \frac{1}{\xi_i(1-\epsilon_i)}k^\p_{i-1}F^\p_{i-1}(A)\v0.
\end{equation}
Applying Theorem \ref{MainThm;JHL}(i) to (\ref{RHS;C+i-1}) we find
\begin{align}
\hat{C}^+_{i-1} 
\nonumber
& =   \frac{\epsilon_i}{\epsilon_i-1}k_iF_i\Big(h(sq)^{1/2}\A+(\tht_0-h-hsq)\Big)v_0 \\
\nonumber
&  \quad  + \frac{1}{\xi_i(1-\epsilon_i)}k^\p_{i-1}F^\p_{i-1}\Big(h(sq)^{1/2}\A+(\tht_0-h-hsq)\Big)\v0 \\
\nonumber
(\textrm{by Lemma \ref{rel;F,p} and Lemma \ref{rel;F,p;perp}}) & =  
\frac{\epsilon_i}{\epsilon_i-1}k_ip_i(\A)v_0 + \frac{1}{\xi_i(1-\epsilon_i)}k^\p_{i-1}p^\p_{i-1}(\A)\v0 \\
\label{C+;v0}
(\textrm{by Lemma \ref{g(v0)=v0p}}) & =  \left(\frac{\epsilon_i}{\epsilon_i-1}k_ip_i(\A) + \frac{1}{\xi_i(1-\epsilon_i)}k^\p_{i-1}p^\p_{i-1}(\A)g[\Y]\right)v_0,
\end{align}
where $\A=\Y+\Y^{-1}$. 
Similarly we find
\begin{equation}\label{C-;v0}
\hat{C}^-_{i} = \left(\frac{1}{1-\epsilon_i}k_ip_i(\A) + \frac{1}{\xi_i(\epsilon_i-1)}k_{i-1}^\p p^\p_{i-1}(\A)g[\Y]\right)v_0.
\end{equation}
Motivated by (\ref{C+;v0}) and (\ref{C-;v0}) we make a definition as follows.

\begin{definition}\label{e;s,r,h}
For $1 \leq i \leq D-1$, define 
\begin{align*}
\varepsilon^{+}_{i-1}[y] &:= \frac{\epsilon_i}{\epsilon_i-1}k_ip_i(y+y^{-1}) + \frac{1}{\xi_i(1-\epsilon_i)}k^\p_{i-1}p^\p_{i-1}(y+y^{-1}) g[y], \\
\varepsilon^-_i[y] &:= \frac{1}{1-\epsilon_i}k_ip_i(y+y^{-1}) + \frac{1}{\xi_i(\epsilon_i-1)}k_{i-1}^\p p^\p_{i-1}(y+y^{-1})g[y].
\end{align*}
Moreover, we define
$$
\varepsilon^-_0 := 1, \qquad \varepsilon^+_{D-1} := k_Dp_D.
$$
\end{definition}

\begin{remark} 
With reference to Defintion \ref{e;s,r,h},
the Laurent polynomials $\{\varepsilon^{\pm}_i\}^{D-1}_{i=0}$ 
are considered as {\it nonsymmetric $q$-Racah polynomials}
in a view of Remark \ref{q-Racah;poly}. The explicit forms are as follows.
For $1 \leq i \leq D-1$
\begin{align*}
\varepsilon^+_{i-1}  &:=  
\tfrac{(1-q^i)(1-s^*q^{D+i+1})(q^{-D},r_1q,r_2q,s^*q;q)_i}{s^iq^i(1-q^D)(1-s^*q)(q,s^*q/r_1,s^*q/r_2,s^*q^{D+2};q)_i} 
{_4}\phi_3
\left(
\begin{matrix}
q^{-i},  s^*q^{i+1},  (sq)^{1/2}y,   (sq)^{1/2}y^{-1} \\
q^{-D},  r_1q,  r_2q
\end{matrix}
~\middle| ~ q, q
\right) \\
& +  \tfrac{s^{1-i}q^{2-2i}(q^{-D+2},r_1q^2,r_2q^2,s^*q^3;q)_{i-1}g[y]}{(1-q^{-D})(1-s^*q^3)(q,s^*q^2/r_1,s^*q^2/r_2,s^*q^{D+2};q)_{i-1}}
{_4}\phi_3
\left(
\begin{matrix}
q^{-i+1},  s^*q^{i+2},  (sq)^{1/2}qy,  (sq)^{1/2}qy^{-1} \\
q^{-D+2},  r_1q^2,  r_2q^2
\end{matrix}
~\middle| ~ q, q
\right), \\
\varepsilon^-_{i}  &:=  
\tfrac{(1-q^{i-D})(1-s^*q^{i+1})(q^{-D},r_1q,r_2q,s^*q;q)_i}{s^iq^i(1-q^{-D})(1-s^*q)(q,s^*q/r_1,s^*q/r_2,s^*q^{D+2};q)_i}
{_4}\phi_3
\left(
\begin{matrix}
q^{-i},  s^*q^{i+1},  (sq)^{1/2}y,   (sq)^{1/2}y^{-1} \\
q^{-D},  r_1q,  r_2q
\end{matrix}
~\middle| ~ q, q
\right) \\
& - \tfrac{s^{1-i}q^{2-2i}(q^{-D+2},r_1q^2,r_2q^2,s^*q^3;q)_{i-1}g[y]}{(1-q^{-D})(1-s^*q^3)(q,s^*q^2/r_1,s^*q^2/r_2,s^*q^{D+2};q)_{i-1}}
{_4}\phi_3
\left(
\begin{matrix}
q^{-i+1},  s^*q^{i+2},  (sq)^{1/2}qy,  (sq)^{1/2}qy^{-1} \\
q^{-D+2},  r_1q^2,  r_2q^2
\end{matrix}
~\middle| ~ q, q
\right),
\end{align*}
where $g[y]$ is from (\ref{poly;g}).
\end{remark}

\begin{proposition}\label{Action;e;C} 
Referring to the $\H$-module $\W(s,s^*,r_1,r_2,D;q)$, for $0 \leq i \leq D-1$
\begin{equation*}
\varepsilon^-_i[\Y].v_0 = \hat{C}^-_i, \qquad \qquad \varepsilon^+_i[\Y].v_0=\hat{C}^+_i.
\end{equation*}
\end{proposition}
\begin{proof}
By (\ref{C+;v0}) and (\ref{C-;v0}) along with Definition \ref{e;s,r,h}, the result follows.
\end{proof}
\noindent
By Proposition \ref{Action;e;C} we can see that the Laurent polynomials $\varepsilon^+_i, \varepsilon^-_i$ play a role of a map that sends $v_0$ to each $\hat{C}^+_i, \hat{C}^-_i$.
Recall the scalars $a,b,c,d$ from Definition \ref{abcd}. 
We will express the nonsymmetric Laurent polynomials $\{\varepsilon^{\pm}_i\}^{D-1}_{i=0}$ in terms of the scalars $a,b,c,d$. 
To this end, we need some preliminary lemmas.

\begin{lemma} Recall the Laurent polynomial $g[y]$ from {\rm (\ref{poly;g})}. Then
$$
g[y] =\frac{b(1-abcd)(1-abcdq)}{(1-bc)(1-bd)}y(1-ay^{-1})(1-by^{-1}).
$$
\end{lemma}
\begin{proof}
Apply the relations (\ref{rels;abcd}) to (\ref{poly;g}). The result follows.
\end{proof}
\begin{lemma} \label{e;x;k}
The following hold.
\begin{enumerate}
\item[\rm (i)] For $1 \leq i \leq D-1$,
\begin{align*}
& \displaystyle \frac{\epsilon_i}{\epsilon_i-1}=\frac{ab(1-q^i)(1-cdq^{i-1})}{(ab-1)(1-abcdq^{2i-1})}, &&
\frac{1}{1-\epsilon_i}= \frac{(1-abq^i)(1-abcdq^{i-1})}{(1-ab)(1-abcdq^{2i-1})}, \\
& \frac{1}{\xi_i(\epsilon_i-1)}=\frac{q^{i-1}}{(ab-1)(1-abcdq^{2i-1})}, &&
\frac{1}{\xi_i(1-\epsilon_i)}=\frac{q^{i-1}}{(1-ab)(1-abcdq^{2i-1})}.
\end{align*}

\item[\rm (ii)] For $0 \leq i \leq D-1$,
$$
k_i = \frac{(ab,ac,ad;q)_i(abcd;q)_{2i}}{a^{2i}(q,bc,bd,cd,abcdq^{i-1};q)_i}.
$$
For $i=D$,
$$
k_D = \frac{(ab,ac,ad;q)_D(abcd;q)_{2D-1}}{a^{2D}(q,bc,bd,abcdq^{D-1};q)_D(cd;q)_{D-1}}.
$$

\item[\rm(iii)] For $0 \leq i \leq D-2$,
$$
k^\p_i = \frac{(abq^2,acq,adq;q)_i(abcdq^2;q)_{2i}}{(aq)^{2i}(q,bcq,bdq,cd,abcdq^{i+1};q)_i}.
$$
\end{enumerate}
\end{lemma}
\begin{proof}
(i): Use (\ref{rels;abcd}) and (\ref{xi;epsilon}).\\
(ii), (iii): Use (\ref{rels;abcd}) and Lemma \ref{k;kp;nu;nup} (i), (ii).
\end{proof}

Recall the polynomial sequences $\{p_i\}^D_{i=0}$ and $\{p_i^\p\}^{D-2}_{i=0}$ from (\ref{AWpoly;ab;finite}) and (\ref{AWpoly;perp;finite}), respectively. Denote their monic polynomials by
\begin{equation}\label{monic;P;Pp}
P_i = \frac{(ab,ac,ad;q)_i}{a^i(abcdq^{i-1};q)_i}p_i , \qquad \qquad 
P^\p_i = \frac{(abq^2,acq,adq;q)_i}{(aq)^i(abcdq^{i+1};q)_i}p^\p_i.
\end{equation}
One checks that $P^\p_i=P_i[y;aq,bq,c,d\mid q]$.

\begin{lemma}\label{e;xi;ki;kip}
\begin{itemize}
\item[\rm (i)] 
For $0\leq i \leq D-1$,
$$ k_ip_i = \frac{(abcd;q)_{2i}}{a^i(q,bc,bd,cd;q)_i}P_i.$$
For $i=D$,
$$
k_Dp_D = \frac{(abcd;q)_{2D-1}}{a^D(q,bc,bd;q)_D(cd;q)_{D-1}}P_D.
$$

\item[\rm (ii)] For $0 \leq i \leq D-2$,
$$
k_i^\p p^\p_i = \frac{(abcdq^2;q)_{2i}}{a^iq^i(q,bcq,bdq,cd;q)_i}P^\p_i.
$$
\end{itemize}
\end{lemma}
\begin{proof}
Use Lemma \ref{e;x;k} (ii), (iii) and (\ref{monic;P;Pp}).
\end{proof}

\begin{proposition}\label{e;abcd}
Let $\{\varepsilon^{\pm}_i\}^{D-1}_{i=0}$ be as in Definition {\rm\ref{e;s,r,h}}. 
Referring to the scalars $a,b,c,d$ associated with $p(\Phi)$, 
the $\{\varepsilon^{\pm}_i\}^{D-1}_{i=0}$ are equal to the following: 
For $1 \leq i \leq D-1$,
\begin{align*}
\varepsilon^+_{i-1} & = \frac{ab(1-q^i)(1-cdq^{i-1})}{(ab-1)(1-abcdq^{2i-1})} 
\frac{(abcd;q)_{2i}}{a^i(q,bc,bd,cd;q)_i}\left(P_i - y(1-ay^{-1})(1-by^{-1})P^\p_{i-1}\right),\\
\nonumber
\varepsilon^-_{i} & = \frac{(1-abq^i)(1-abcdq^{i-1})}{(1-ab)(1-abcdq^{2i-1})}
\frac{(abcd;q)_{2i}}{a^i(q,bc,bd,cd;q)_i} \\
& \qquad \times \left(P_i - \frac{ab(1-q^i)(1-cdq^{i-1})}{(1-abq^i)(1-abcdq^{i-1})}y(1-ay^{-1})(1-by^{-1})P^\p_{i-1}\right).
\end{align*}
Moreover,
$$
\varepsilon^-_0 = 1, \qquad \qquad 
\varepsilon^+_{D-1} = \frac{(abcd;q)_{2D-1}}{a^D(q,bc,bd;q)_D(cd;q)_{D-1}}P_D.
$$
\end{proposition}
\begin{proof}
Apply Lemma \ref{e;x;k} and Lemma \ref{e;xi;ki;kip} to Definition \ref{e;s,r,h}. The result follows.
\end{proof}

\noindent
We give a comment. For $1 \leq i \leq D-1$, the $\varepsilon^+_{i-1}$ has of the form
\begin{align*}
\frac{ab(1-q^i)(1-cdq^{i-1})}{(ab-1)(1-abcdq^{2i-1})} 
\frac{(abcd;q)_{2i}}{a^i(q,bc,bd,cd;q)_i}
\Big( ({\rm constant}) y^{i-1} + \cdots + (1-ab)y^{-i} \Big),
\end{align*}
and the $\varepsilon^-_i$ has of the form
\begin{align*}
\frac{(abcd;q)_{2i}}{a^i(q,bc,bd,cd;q)_i}
\left( y^i + \cdots
+ \frac{1+ab-abq^i-abcdq^{i-1}}{1-abcdq^{2i-1}}y^{-i} \right).
\end{align*}
Therefore the set $\{\varepsilon^{\pm}_i\}^{D-1}_{i=0}$ is linearly independent in $\C[y,y^{-1}]$.

\begin{remark}\label{H-mod;L}
Let $L$ denote the subspace of $\C[y,y^{-1}]$ spanned by $\{\varepsilon^{\pm}_i\}^{D-1}_{i=0}$.
By the above comment, $\{\varepsilon^{\pm}_i\}^{D-1}_{i=0}$ are a basis for $L$.
Recall the $\H$-module $\W(a,b,c,d;q)$ in Lemma \ref{Hq-mod;abcd}.
The space $L$ is isomorphic to $\W$ via a $\C$-vector space isomorphism that 
sends $\varepsilon^{\sigma}_i$ to $\hat{C}^{\sigma}_i$ for $\sigma \in \{+,-\}$
and $i=0,1,\ldots , D-1$.
By these comments we can endow a module structure for $\H$ with $L$. 
On this module $L$, the matrix representing $t_n$ relative to a basis $\varepsilon^-_0,\varepsilon^+_0, \varepsilon^-_0,\varepsilon^+_0, \ldots, \varepsilon^-_{D-1},\varepsilon^+_{D-1}$ coincides with the matrix ${\bf T}_n$ in Lemma \ref{Hq-mod;abcd}.
We denote this module by $L(a,b,c,d;q)$.
\end{remark}

\section{The operator $\Y$}

In Section \ref{Section;UDAHA} we mentioned the eigenvalues/ eigenvectors of $\X$ on $\W(s,s^*,r_1,r_2,D;q)$. 
In this section we will find eigenvectors of $\Y$ along with the corresponding eigenvalues on $\W(s,s^*,r_1,r_2,D;q)$.
First we find the eigenvalues of $\A=\Y+\Y^{-1}$. Throughout this section we work on the $\H$-module $\W=\W(s,s^*,r_1,r_2,D;q)$.
\begin{lemma}\label{e-val(A)}
The eigenvalues of $\A$ are
 $q^i(sq)^{1/2}+q^{-i}(sq)^{-1/2}$ for  $0 \leq i \leq D$.
\end{lemma}
\begin{proof}
By Theorem \ref{MainThm;JHL}(i), we find the equation 
$\A={h^{-1}(sq)^{-1/2}} \left( A-(\tht_0-h-hsq)I \right)$ on $\W$. 
Recall that $A$ has the eigenvalues $\{\tht_i\}^D_{i=0}$ and 
each $\tht_i$ has the form (\ref{theta_i}). 
By these comments, the result follows.
\end{proof}
For notational convenience we denote $\ell_i=q^i(sq)^{1/2}+q^{-i}(sq)^{-1/2}$ for $0 \leq i \leq D$. 
Let $W_{\ell_i}$ denote the eigenspace of $\A$ for $\ell_i$.
Then $W_{\ell_i}=E_i\W$ for $0 \leq i \leq D$ since $A$ and $\A$ share a common eigenvector by Theorem \ref{MainThm;JHL}(i).
Observe that $\W=\sum^D_{i=0} W_{\ell_i}$, the orthogonal direct sum.
Moreover, by construction of the $T$-module $\W$ the following lemma holds.
\begin{lemma}\label{E-spaceA} For $1 \leq i \leq D-1$, we have
$$
W_{\ell_i} = {\rm span}\{E_iv_0, E_{i-1}^\p v_0^\p\}.
$$
Moreover, $W_{\ell_0} = {\rm span}\{ E_0v_0 \}$ 
and $W_{\ell_D} = {\rm span}\{ E_Dv_0\}.$
\end{lemma}

By Lemma \ref{kn;diagonalizable} each of $t_0$ and $t_1$ is diagonalizable. 
By \cite[Lemma 3.8]{IT} $t_0$ and $t_1$ commute with $\A=t_0t_1 + (t_0t_1)^{-1}$, 
so each of $t_0$ and $t_1$ shares the eigenspaces of $\A$. 
It follows that $W_{\ell_i}(0 \leq i \leq D)$ is invariant under $t_n (n=0,1)$. 
By these comments and Lemma \ref{E-spaceA} 
the matrix representing $t_n (n=0,1)$ relative to the ordered basis
$$
\mcal{B} := \{ E_0v_0, E_1v_0, E_0^\p v_0^\p, E_2v_0,E_1^\p v_0^\p, \ldots, E_{D-1}v_0,E_{D-2}^\p v_0^\p,E_Dv_0 \}
$$
for $\W$ takes the form
\begin{equation}\label{MatForm}
\scalemath{0.7}{
\begin{pmatrix}
* &&&&&&& \vspace {-0.2cm}\\
&*&*&&&&& \\
&*&*&&&&& \vspace {-0.2cm}\\
&&&*&*&&& \\
&&&*&*&&& \vspace {-0.2cm}\\
&&&&&\ddots &&\vspace {-0.2cm}\\
&&&&&&*&*& \\
&&&&&&*&*& \vspace {-0.2cm}\\
&&&&&&&&* \\
\end{pmatrix}.}
\end{equation}
We explicitly find the matrix representing $t_n (n=0,1)$ relative to $\mcal{B}$,
in order to find the eigenvectors of $\Y$.
First we find the eigenvectors of $t_0$ in terms of vectors in $\mcal{B}$. 
For notational convenience we define $\hat{C}_{-1}^{-}=0, \hat{C}_{-1}^{+}=0$ and $\hat{C}^{-}_D=0, \hat{C}^{+}_D=0$.

\begin{lemma}\label{e-vec;t0(W)}
For $1 \leq i \leq D-1$,
\begin{align}
\label{e-vector;k0}
&t_0.(\hat{C}^+_{i-1}+\hat{C}^-_i) = \k_0(\hat{C}^+_{i-1}+\hat{C}^-_i), \\
\label{e-vector;k0-1}
&t_0.(\epsilon^{-1}_i \hat{C}^+_{i-1}+\hat{C}^-_i)=\k_0^{-1}(\epsilon^{-1}_i \hat{C}^+_{i-1}+\hat{C}^-_i),
\end{align}
where $\epsilon_i$ is from {\rm (\ref{xi;epsilon})} and $\k_0$ is from {\rm (\ref{k_n})}.
\end{lemma}

\begin{proof}
To show (\ref{e-vector;k0}) it suffices to show that $(t_0-\k_0).(\hat{C}^+_{i-1}+\hat{C}^-_i) = 0$.
By Theorem \ref{MainThm;JHL}(iv) the 
$(t_0-\k^{-1}_0)(\k_0-\k^{-1}_0)^{-1}$ acts as the projection onto $\Mx$. 
By this and since $\hat{C}^+_{i-1}+\hat{C}^-_i=v_i \in \Mx$, we find (\ref{e-vector;k0}). 
Next we show (\ref{e-vector;k0-1}).
Similar to (\ref{e-vector;k0}), it suffices to show 
$(t_0-\k^{-1}_0).(\epsilon^{-1}\hat{C}^+_{i-1}+\hat{C}^-_i) = 0$.
Since $(t_0-\k_0)(\k^{-1}_0-\k_0)^{-1}=1-(t_0-\k^{-1}_0)(\k_0-\k^{-1}_0)^{-1}$ 
acts as the projection onto $\Mxp$ and 
$\epsilon_i^{-1}\hat{C}^+_{i-1}+\hat{C}^-_i = \epsilon_i^{-1}\xi^{-1}_iv^\p_{i-1} \in \Mxp$ 
for $1 \leq i \leq D-1$, the desired result follows.
\end{proof}
\noindent
Consider the eigenvector $\hat{C}^+_0+\hat{C}^-_1$ of $t_0$ for $\k_0$. Observe that
$\hat{C}^+_0+\hat{C}^-_1=v_1=Av_0=\sum^D_{r=0}\tht_rE_rv_0,$
where the last equation is from $A=\sum^D_{r=0}\tht_rE_r$.
Therefore the coordinate vector of $\hat{C}^+_0+\hat{C}^-_1$ relative to $\mcal{B}$
is
\begin{equation}\label{e-vec_t0_k0;i=1}
\big[ \tht_0, \tht_1, 0, \tht_2, 0, \ldots, \tht_{D-1},0,\tht_D \big]^t.
\end{equation}
\begin{lemma}\label{1-dim;t0,D}
For $i \in \{0, D\}$, 
$$
t_0.E_iv_0 = \k_0E_iv_0,
$$
where $\k_0$ is from {\rm (\ref{k_n})}.
\end{lemma}
\begin{proof}
The matrix representing $t_0$ relative to $\mcal{B}$ has the form (\ref{MatForm}). 
This matrix has the eigenvector (\ref{e-vec_t0_k0;i=1}) corresponding to the
eigenvalue $\k_0$.
Use this and linear algebra to get the result.
\end{proof}
\noindent
Next consider the eigenvector 
$\epsilon_1^{-1}\hat{C}^+_0+\hat{C}^-_1$ of $t_0$ for $\k^{-1}_0$.
From the left in (\ref{lem8.3}) for $i=0$, 
\begin{equation}\label{C+0;Er;Eps}
\hat{C}^+_{0} 
= \tfrac{\epsilon_1}{\epsilon_1-1}v_1 + \tfrac{1}{\xi_1(1-\epsilon_1)}v^\p_{0}
=\sum^D_{r=0}\tfrac{\epsilon_1\tht_r}{\epsilon_1-1} E_rv_0 +\sum^{D-2}_{s=0} \tfrac{1}{\xi_1(1-\epsilon_1)} E_s^\p v_0^\p,
\end{equation}
where the last equality holds from 
$v_1 = Av_0=\sum^D_{r=0}\tht_rE_r$ and
$v_0^\p=\sum^{D-2}_{s=0}E^\p_sv_0^\p$.
Similarly, from the right in (\ref{lem8.3}) for $i=1$ we find
\begin{equation}\label{C-0;Er;Eps}
\hat{C}^-_1
=\sum^D_{r=0} \tfrac{\tht_r}{1-\epsilon_1} E_rv_0 + \sum^{D-2}_{s=0} \tfrac{1}{\xi_1(\epsilon_1-1)} E_s^\p v_0^\p.
\end{equation}
Using (\ref{C+0;Er;Eps}) and (\ref{C-0;Er;Eps}) we have
$\epsilon_1^{-1}\hat{C}^+_0+\hat{C}^-_1
= \sum^{D-2}_{s=0} \tfrac{1}{\xi_1\epsilon_1} E_s^\p v_0^\p$.
Therefore the coordinate vector of $\epsilon_1^{-1}\hat{C}^+_0+\hat{C}^-_1$ relative
to $\mcal{B}$ is
\begin{equation}\label{e-vec_t0_k0-;i=1}
\left[ 0, 0, \tfrac{1}{\xi_1\epsilon_1}, 0, \tfrac{1}{\xi_1\epsilon_1}, 0, \tfrac{1}{\xi_1\epsilon_1}, \ldots, 0,\tfrac{1}{\xi_1\epsilon_1},0 \right]^t.
\end{equation}

\begin{lemma}\label{[t0]_B}
The matrix representation of $t_0$ relative to $\mcal{B}$ is
$$
\bdiag\Big( [\k_0], [t_0(1)], [t_0(2)], \ldots, [t_0(D-1)], [\k_0] \Big),
$$
where $[t_0(i)] = \diag(\k_0, \k_0^{-1})$ for $1 \leq i \leq D-1$.
\end{lemma}
\begin{proof} 
Let $[t_0]_{\mcal{B}}$ denote the matrix representation of $t_0$ 
relative to $\mcal{B}$. Since $[t_0]_{\mcal{B}}$ has the form (\ref{MatForm})
and by Lemma \ref{1-dim;t0,D} we denote
$[t_0]_{\mcal{B}}$ by $\bdiag\left( [\k_0], [t_0(1)],  \ldots,[t_0(D-1)], [\k_0] \right),$
where the $[t_0(i)]$ is the matrix representing $t_0$ relative to $\{E_iv_0, E^\p_{i-1}v_0^\p\}$ for $1 \leq i \leq D-1$.
Using linear algebra along with (\ref{e-vec_t0_k0;i=1}) and (\ref{e-vec_t0_k0-;i=1})
we find the equation
\begin{equation*}\label{[t0(i)];equation}
[t_0(i)]
\begin{bmatrix}
\theta_i & 0 \\
0 & \frac{1}{\xi_1\epsilon_1}
\end{bmatrix}=
\begin{bmatrix}
\theta_i & 0 \\
0 & \frac{1}{\xi_1\epsilon_1}
\end{bmatrix}
\begin{bmatrix}
\k_0 & 0 \\
0 & \k_0^{-1}
\end{bmatrix}.
\end{equation*}
Evaluate $[t_0(i)]$ using the above equation. 
The result follows.
\end{proof}
\noindent

We now find the matrix representation of $t_1$ relative to $\mcal{B}$. 
We start with the following lemma.
\begin{lemma}\label{e-vec;t1} For $0 \leq i \leq D-1$,
\begin{align*}
& t_1.(\hat{C}^-_i+\hat{C}^+_i) = \k_1(\hat{C}^-_i+\hat{C}^+_i), \\
& t_1.(\tau_i\hat{C}^-_i+\hat{C}^+_i) = \k^{-1}_1(\tau_i\hat{C}^-_i+\hat{C}^+_i),
\end{align*}
where $\k_1$ is from {\rm (\ref{k_n})} and 
\begin{equation*}
\tau_i=\frac{s^*(1-r_1q^{i+1})(1-r_2q^{i+1})}{(r_1-s^*q^{i+1})(r_2-s^*q^{i+1})}.
\end{equation*}
\end{lemma}
\begin{proof}
It is similar to Lemma \ref{e-vec;t0(W)}.
Use Theorem \ref{MainThm;JHL}(v) and the fact that
$\{\hat{C}^-_i+\hat{C}^+_i\}^{D-1}_{i=0}$ is a basis for $\MC$
and $\{\tau_i\hat{C}^-_i+\hat{C}^+_i\}^{D-1}_{i=0}$ is a basis for $\MC^\p$.
\end{proof}

\noindent
Consider the eigenvector $\hat{C}^-_0+\hat{C}^+_0$ of $t_1$ for $\k_1$. 
Since $\hat{C}^-_0=v_0=\sum^D_{r=0}E_rv_0$,
by this and (\ref{C+0;Er;Eps}) we have
\begin{equation*}
\hat{C}^-_0+\hat{C}^+_0  = v_0 + \tfrac{\epsilon_1}{\epsilon_1-1}v_1 + \tfrac{1}{\xi_1(1-\epsilon_1)}v_0^\p 
= \sum^D_{r=0}\left( 1+\tfrac{\epsilon_1\tht_r}{\epsilon_1-1} \right)E_rv_0 + \sum^{D-2}_{s=0} \tfrac{1}{\xi_1(1-\epsilon_1)}E^\p_sv_0^\p.
\end{equation*}
One routinely checks that $1+\frac{\epsilon_1\tht_D}{\epsilon_1-1}=0$, and so the coordinate vector of $\hat{C}^-_0+\hat{C}^+_0$ relative to $\mcal{B}$ is
\begin{equation}\label{e-vec;t1;k1;i=0}
\left[ 1+\tfrac{\epsilon_1\tht_0}{\epsilon_1-1}, 1+\tfrac{\epsilon_1\tht_1}{\epsilon_1-1}, \tfrac{1}{\xi_1(1-\epsilon_1)},1+\tfrac{\epsilon_1\tht_2}{\epsilon_1-1},\tfrac{1}{\xi_1(1-\epsilon_1)}, \ldots, 1+\tfrac{\epsilon_1\tht_{D-1}}{\epsilon_1-1},\tfrac{1}{\xi_1(1-\epsilon_1)},0  \right]^t.
\end{equation}
Similarly, for the eigenvector $\tau_0\hat{C}^-_0+\hat{C}^+_0$ of $t_1$ for $k_1^{-1}$ we find
\begin{equation*}
\tau_0\hat{C}^-_0+\hat{C}^+_0 = \sum^D_{r=0}\left( \tau_0+\tfrac{\epsilon_1\tht_r}{\epsilon_1-1} \right)E_rv_0 + \sum^{D-2}_{s=0} \tfrac{1}{\xi_1(1-\epsilon_1)}E^\p_sv_0^\p.
\end{equation*}
One routinely checks that $\tau_0+\frac{\epsilon_1\tht_0}{\epsilon_1-1}=0$, and so the coordinate vector of $\tau_0\hat{C}^-_0+\hat{C}^+_0$ relative to $\mcal{B}$ is
\begin{equation}\label{e-vec;t1;k1-;i=0}
\left[ 0,\tau_0+\tfrac{\epsilon_1\tht_1}{\epsilon_1-1}, \tfrac{1}{\xi_1(1-\epsilon_1)}, \tau_0+\tfrac{\epsilon_1\tht_2}{\epsilon_1-1}, \tfrac{1}{\xi_1(1-\epsilon_1)}, \ldots, \tau_0+\tfrac{\epsilon_1\tht_{D-1}}{\epsilon_1-1},\tfrac{1}{\xi_1(1-\epsilon_1)},\tau_0+\tfrac{\epsilon_1\tht_D}{\epsilon_1-1}  \right]^t.
\end{equation}

\begin{lemma}\label{e-vec;t1;1-dim}
We have
$$
t_1.E_0v_0 = \k_1E_0v_0, \qquad \qquad t_1.E_Dv_0=\k_1^{-1}E_Dv_0.
$$
\end{lemma}
\begin{proof}
The matrix representing $t_1$ relative to $\mcal{B}$ has the form (\ref{MatForm}). 
Use this and linear algebra together with (\ref{e-vec;t1;k1;i=0}), (\ref{e-vec;t1;k1-;i=0}).
The result follows.
\end{proof}
\begin{lemma}\label{[t1]_B}
The matrix representation of $t_1$ relative to $\mcal{B}$ is
$$
\bdiag\Big( [\k_1], [t_1(1)], [t_1(2)], \ldots, [t_1(D-1)], [\k_1^{-1}] \Big),
$$
where for $1 \leq i \leq D-1$
$$
[t_1(i)] = 
\begin{bmatrix}
\k^{-1}_1\left( \frac{\k_1^2-\tau_0}{1-\tau_0}-\theta_i\frac{\epsilon_1(\k_1^2-1)}{(1-\epsilon_1)(1-\tau_0)} \right) &
\frac{\k^{-1}_1\xi_1(1-\k_1^2)(\epsilon_1\theta_i-1+\epsilon_1)(\epsilon_1\theta_i-\tau_0+\tau_0\epsilon_1)}{(1-\epsilon_1)(1-\tau_0)}\\
 \frac{\k^{-1}_1(\k^2_1-1)}{\xi_1(1-\epsilon_1)(1-\tau_0)}&
\k^{-1}_1\left( \theta_i\frac{\epsilon_1(\k_1^2-1)}{(1-\epsilon_1)(1-\tau_0)} - \frac{k_1^2\tau_0-1}{1-\tau_0} \right)
\end{bmatrix}.
$$
\end{lemma}
\begin{proof}
Similar to the proof of Lemma \ref{[t0]_B}.
\end{proof}
\noindent
Based on our discussion so far, we find the matrix representation of $\Y$ relative to $\mcal{B}$.
\begin{lemma}\label{[Y]_B}
The matrix representing $\Y$ relative to $\mcal{B}$ is
$$
\bdiag\Big( [(sq)^{1/2}], [\Y(1)], [\Y(2)], \ldots [\Y(D-1)], [q^{-D}(sq)^{-1/2}] \Big),
$$
where for $1\leq i \leq D-1$
\begin{equation}\label{[Y(i)]}
[\Y(i)] =
\begin{bmatrix}
\k_0\k^{-1}_1\left( \frac{\k_1^2-\tau_0}{1-\tau_0}-\theta_i\frac{\epsilon_1(\k_1^2-1)}{(1-\epsilon_1)(1-\tau_0)} \right) &
\k_0\k^{-1}_1\frac{\xi_1(1-\k_1^2)(\epsilon_1\theta_i-1+\epsilon_1)(\epsilon_1\theta_i-\tau_0+\tau_0\epsilon_1)}{(1-\epsilon_1)(1-\tau_0)}\\
\k_0^{-1}\k^{-1}_1 \frac{(\k^2_1-1)}{\xi_1(1-\epsilon_1)(1-\tau_0)}&
\k_0^{-1}\k^{-1}_1\left( \theta_i\frac{\epsilon_1(\k_1^2-1)}{(1-\epsilon_1)(1-\tau_0)} - \frac{k_1^2\tau_0-1}{1-\tau_0} \right)
\end{bmatrix}.
\end{equation}
\end{lemma}
\begin{proof}
Recall $\Y=t_0t_1$. The matrix representing $\Y$ relative to $\mcal{B}$ is 
the product of the matrix representing $t_0$ relative to $\mcal{B}$ and 
the matrix representing $t_1$ relative to $\mcal{B}$. By Lemma \ref{[t0]_B} 
and Lemma \ref{[t1]_B} the result follows.
\end{proof}

\begin{lemma}
The eigenvalues of $\Y$ are
$$
\begin{array}{llllllllll}
(sq)^{1/2}, & q(sq)^{1/2}, & q^2(sq)^{1/2}, & \ldots , & q^{D-1}(sq)^{1/2}, \\
& q^{-1}(sq)^{-1/2}, & q^{-2}(sq)^{-1/2}, & \ldots , & q^{1-D}(sq)^{-1/2}, &q^{-D}(sq)^{-1/2}.
\end{array}
$$
Therefore $\Y$ is multiplicity-free.
\end{lemma}
\begin{proof}
The eigenvalues of $\Y$ are routinely obtained from Lemma \ref{[Y]_B}.
The last assertion follows from assumption (ii) in Example \ref{ex;q-racah;PA}.
\end{proof}

For notational convenience we abbreviate  
$$
\lambda_r = 
\begin{cases}
 q^r(sq)^{1/2} & (r=0,1,2,\ldots,D-1), \\
 q^r(sq)^{-1/2} & (r=-1,-2, \ldots, -D).
\end{cases}
$$
Note that  $\lambda^{-1}_i = \lambda_{-i}$ for $i=\pm1,\pm2, \ldots, \pm(D-1)$ 
and $\ell_j = \lambda_j+\lambda_{-j}$ for $0 \leq j \leq D$.
We will find an eigenvector of $\Y$ associated with $\lambda_r (-D \leq r \leq D-1)$, and express as a linear combination of elements of $\mcal{B}$. 

\begin{lemma}\label{pre:e-vec;Y}
With the above notation, 
there exist an eigenvector $\overline{\y}_i$ of $\Y$ associated with the eigenvalue
$\lambda_i$ such that for $i\in \{0,D\}$ 
$$
\overline{\y}_0 = E_0v_0, \qquad \qquad \overline{\y}_{-D}= E_Dv_0,
$$
and for $1 \leq i \leq D-1$
\begin{equation*}
\overline{\y}_{-i} = \omega_{-i}E_iv_0 + E^\p_{i-1}v_0^\p, 
\qquad \qquad 
\overline{\y}_{i} = \omega_iE_iv_0 + E^\p_{i-1}v_0^\p,
\end{equation*}
where
\begin{align}
\nonumber
\omega_{-i} & = m\left(q^{-i}+\frac{\theta_i}{h(q^{-D}-1)} + \frac{r_1r_2(1-r_1q)(1-r_2q)-(r_1-s^*q)(r_2-s^*q)}{(r_1r_2-s^*)(1-s^*q^2)}\right),\\
\nonumber
\omega_{i} &= m\left(sq^{i+1}+\frac{\theta_i}{h(q^{-D}-1)} + \frac{r_1r_2(1-r_1q)(1-r_2q)-(r_1-s^*q)(r_2-s^*q)}{(r_1r_2-s^*)(1-s^*q^2)}\right),
\end{align}
and $m=\tfrac{s^*(1-q^{-D})(1-s^*q^2)(1-s^*q^3)}{(r_1-s^*q)(r_2-s^*q)}$.

\end{lemma}
\begin{proof}
Note that $\A$ and $\Y$ share common eigenvectors.
Without loss of generality one can choose 
$\overline{\y}_0 = E_0v_0$ and $\overline{\y}_{-D}=E_Dv_0$.
Let $1 \leq i \leq D-1$. Since $\A=\Y+\Y^{-1}$, each of $\overline{\y}_i, \overline{\y}_{-i}$ is an eigenvector of $\A$ associated with $\ell_i$. 
So $\overline{\y}_i, \overline{\y}_{-i}\in W_{\ell_i}$.
Let $[\overline{\y}_i]$ (resp. $[\overline{\y}_{-i}]$) denote the coordinate vector of $\overline{\y}_i$ (resp. $\overline{\y}_{-i}$) relative to $\{E_iv_0, E_{i-1}^{\p}v_0^\p\}$. 
It suffices to find the vectors $[\overline{\y}_i]$ and $[\overline{\y}_{-i}]$. 
By Lemma \ref{[Y]_B} we have
$$
[\Y(i)][\overline{\y}_i] = \lambda_i[\overline{\y}_i] \qquad \text{\rm and} \qquad [\Y(i)][\overline{\y}_{-i}] = \lambda_{-i}[\overline{\y}_{-i}].
$$ 
Evaluate the above equations using (\ref{[Y(i)]}) and 
simplify the vectors $[\overline{\y}_{i}],[\overline{\y}_{-i}]$. The result follows.
\end{proof}
\noindent
We normalize the vectors $\{\overline{\y}_i\}^{D-1}_{i=-D}$ so that the sum of these vectors is equal to $v_0$. 
For $1 \leq i \leq D-1$ set
\begin{align}
\label{ev;y;-i}
&\y_{-i}  := -\frac{q^i}{m(sq^{2i+1}-1)}\overline{\y}_{-i} = -\frac{q^i\omega_{-i}}{m(sq^{2i+1}-1)}E_iv_0 
- \frac{q^i}{m(sq^{2i+1}-1)}E^\p_{i-1}v^\p_0, \\
\label{ev;y;i}
&\y_i  := \frac{q^i}{m(sq^{2i+1}-1)}\overline{\y}_i = \frac{q^i\omega_i}{m(sq^{2i+1}-1)}E_iv_0 + \frac{q^i}{m(sq^{2i+1}-1)}E^\p_{i-1}v^\p_0.
\end{align}
Set
\begin{equation}\label{ev;y;0,D}
\y_0 := \overline{\y}_0=E_0v_0, \qquad \qquad \y_{-D} := \overline{\y}_{-D}=E_Dv_0.
\end{equation}
\begin{theorem}\label{e-vec;Y}
With reference to the notation {\rm (\ref{ev;y;-i})--(\ref{ev;y;0,D})},
each $\y_i$ is an eigenvector of $\Y$ associated with $\lambda_i$,
and
\begin{equation}\label{sum;y;v0}
\sum^{D-1}_{i=-D}\y_i = v_0.
\end{equation}
\end{theorem}
\begin{proof} It suffices to show (\ref{sum;y;v0}). 
Evaluate the right-hand side in (\ref{sum;y;v0}) using Lemma \ref{pre:e-vec;Y} together with (\ref{ev;y;-i})--(\ref{ev;y;0,D}). Then
$$
\sum^{D-1}_{i=-D}\y_i  = \y_0+ \sum^{D-1}_{i=0}(\y_i + \y_{-i}) + \y_{-D} 
= E_0v_0 + \sum^{D-1}_{i=0}\frac{q^i(\overline{\y}_i-\overline{\y}_{-i})}{m(sq^{2i+1}-1)} + E_Dv_0
=\sum^D_{i=0}E_iv_0=v_0,
$$
as required.
\end{proof}
\noindent
For the rest of this paper we fix the eigenvectors $\{\y_i\}^{D-1}_{i=-D}$ of $\Y$ 
as in Theorem \ref{e-vec;Y}. Let $W_{\lambda_i}$ denote the eigenspace of 
$\Y$ for $\lambda_i$. Observe that $\W = \bigoplus^{D-1}_{i=-D}W_{\lambda_i}$, 
an orthogonal direct sum. Recall the Hermitian inner product $\<\cdot,\cdot\>_V$
from the first paragraph in Section \ref{S3;DRG}.
We find the norm of $\y_i$ for $-D \leq i \leq D-1$. 

\begin{proposition}\label{norm;y}
For $1\leq i \leq D-1$,
\begin{align*}
& \Vert \y_{-i}\Vert^2 = \frac{q^{2i}\omega^2_{-i}}{m^2(sq^{2i+1}-1)^2}k^*_i\nu^{-1} +  
\frac{q^{2i}}{m^2(sq^{2i+1}-1)^2} k_{i-1}^{*\p}\nu^{\p-1}\Vert v_0^\p \Vert^2,\\
& \Vert \y_i \Vert^2 = \frac{q^{2i}\omega^2_{i}}{m^2(sq^{2i+1}-1)^2}k^*_i\nu^{-1} + 
\frac{q^{2i}}{m^2(sq^{2i+1}-1)^2}k_{i-1}^{*\p}\nu^{\p-1}\Vert v_0^\p\Vert^2,
\end{align*}
where $k^*_i$, $k^{*\p}_i$, $\nu$, $\nu^\p$ are from Lemma {\rm\ref{k;kp;nu;nup}},
$m$ is from Lemma {\rm\ref{pre:e-vec;Y}}, and 
\begin{equation} \label{xp;norm} 
\begin{split}
\Vert v_0^\p\Vert^2 = 
\frac{s^*(1-q^D)(1-q^{1-D})(1-s^*q^2)(1-s^*q^3)(1-r_1q)(1-r_2q)}{q^Dr_1r_2(1-s^*q/r_1)(1-s^*q/r_2)}.
\end{split}
\end{equation}
Moreover,
$$
\Vert \y_0\Vert^2 = \nu^{-1}, \qquad \qquad \Vert \y_{-D}\Vert^2 = k^*_D\nu^{-1}.
$$
\end{proposition}

\begin{proof} 
Recall from \cite[Theorem 15.3]{PT:Madrid} that
\begin{equation}\label{thm15.3}
\< E_iv_0, E_jv_0\>_V =\delta_{ij}k^*_i\nu^{-1}\Vert v_0\Vert^2 \qquad (0 \leq i, j \leq D-1).
\end{equation}
Let $1 \leq i \leq D-1$. We first find $\Vert \y_{-i}\Vert^2$. By (\ref{ev;y;-i}), we have
\begin{align*}
\Vert \y_{-i}\Vert^2 & = \left\< -\frac{q^i\omega_{-i}}{m(sq^{2i+1}-1)}E_iv_0-\frac{q^i}{m(sq^{2i+1}-1)} E_{i-1}^\p v_0^\p,
-\frac{q^i\omega_{-i}}{m(sq^{2i+1}-1)}E_iv_0-\frac{q^i}{m(sq^{2i+1}-1)} E_{i-1}^\p v_0^\p\right\>_V \\
&=\frac{q^{2i}\omega^2_{-i}}{m^2(sq^{2i+1}-1)^2}\Vert E_iv_0\Vert^2 + 
\frac{q^{2i}}{m^2(sq^{2i+1}-1)^2}\Vert E_{i-1}^\p v_0^\p\Vert^2 \\
&= \frac{q^{2i}\omega^2_{-i}}{m^2(sq^{2i+1}-1)^2}k^*_i\nu^{-1}\Vert v_0\Vert^2 + 
\frac{q^{2i}}{m^2(sq^{2i+1}-1)^2}k^{*\p}_{i-1}\nu^{\p-1}\Vert v_0^\p\Vert^2,
\end{align*}
where the last equation is from (\ref{thm15.3}). Recall that $\Vert v_0 \Vert^2=1$. The line (\ref{xp;norm}) is obtained from \cite[Lemma 6.14]{JHL} together with Lemma \ref{norm;C;pm}. Similarly we obtain the norm of $\y_i$.
\end{proof}

\section{Orthogonality for $\varepsilon^{\pm}_i$}\label{section;Orthogonality}
Recall the nonsymmetric Laurent polynomials $\{\varepsilon^{\pm}_i\}^{D-1}_{i=0}$
from Definition \ref{e;s,r,h} and recall the subspace $L$ of $\C[y,y^{-1}]$ 
spanned by $\{\varepsilon^{\pm}_i\}^{D-1}_{i=0}$ from Remark \ref{H-mod;L}.
In this section we define a bilinear form on $L$ that satisfies the
orthogonality relations for $\{\varepsilon^{\pm}_i\}^{D-1}_{i=0}$.
Recall the bilinear form $\< \cdot, \cdot \>_V$ 
and the $\H$-module $\W(s,s^*,r_1,r_2,D;q)$.

\begin{lemma} \label{antiauto;bilinear;W}
Let the antiautomorphism $\dagger$ be as in Lemma {\rm \ref{antiauto;Hq}}. For $h \in \H$ and $u,v \in \W(s,s^*,r_1,r_2,D;q)$,
$$
\<h.u,v\>_V = \<u,h^{\dagger}.v\>_V.
$$
\end{lemma}
\begin{proof}
Let $0 \leq i, j \leq D-1$.
Routinely check 
$$\<t^{\delta}_n.\hat{C}^{\sigma}_i,\hat{C}^{\tau}_j\>_V=\<\hat{C}^{\sigma}_i,t^{\delta \dagger}_n.\hat{C}^{\tau}_j\>_V,$$ 
for $n\in \I$ and $\sigma, \tau, \delta \in \{+,-\}$.
Since $\{t^{\pm1}_n\}_{n \in \I}$ generates $\H$, the result follows.
\end{proof}
\begin{lemma}\label{bilinear;L1,2}
For $-D \leq i \leq D-1$, let $\y_i$ be an eigenvector of $\Y$ for the eigenvalue $\lambda_i$ as in Theorem {\rm \ref{e-vec;Y}}. For $L_1,L_2 \in L$,
$$
\<L_1[\Y].v_0, L_2[\Y].v_0\>_V = \sum^{D-1}_{i=-D}L_1[\lambda_i]\overline{L_2[\lambda_i]}\Vert \y_i\Vert^2,
$$
where the norm $\Vert \y_i \Vert^2$ is given by Proposition {\rm \ref{norm;y}}.
\end{lemma}
\begin{proof}
Compute
\begin{align*}
\<L_1[\Y].v_0, L_2[\Y].v_0\>_V
& = \<L_1[\Y]\sum_i \y_i, L_2[\Y]\sum_j \y_j\>_V && \text{(by line (\ref{sum;y;v0}))}\\
& = \sum_{i,j} \<L_1[\Y] \y_i, L_2[\Y] \y_j\>_V\\
& = \sum_{i,j} \<L_1[\lambda_i] \y_i, L_2[\lambda_j] \y_j\>_V\\
& = \sum_{i,j} L_1[\lambda_i]  \overline{L_2[\lambda_j]} \< \y_i, \y_j\>_V\\
& = \sum_{i}L_1[\lambda_i]\overline{L_2[\lambda_i]}\Vert\y_i\Vert^2,
\end{align*}
where the last equation follows since $\{\y_i\}^{D-1}_{i=-D}$ is an orthogonal basis for $\W$.
\end{proof}

\noindent
Motivated by Lemma \ref{bilinear;L1,2} we define a bilinear form 
$\< \cdot , \cdot \>_L : L \times L \to \C$ as follows. 
For $L_1,L_2 \in L$ 
\begin{equation}\label{bilinear;L}
\< L_1[y],L_2[y]\>_L :=  \sum^{D-1}_{i=-D}L_1[\lambda_i]\overline{L_2[\lambda_i]}\Vert \y_i\Vert^2,
\end{equation}
where $\Vert \y_i\Vert^2$ is from Proposition \ref{norm;y}.

\begin{lemma}\label{bil;property}
Let $L_1,L_2 \in \C[y,y^{-1}]$. With reference to the form {\rm (\ref{bilinear;L})},
\begin{equation}\label{bil;property;pf}
\< yL_1[y],L_2[y]\>_L = \< L_1[y],yL_2[y]\>_L. 
\end{equation}
\end{lemma}
\begin{proof}
On the left-hand side in (\ref{bil;property;pf}) put $L'_1[y]=yL_1[y]$. 
Then by Lemma \ref{bilinear;L1,2}
$$\<L'_1[y],L_2[y]\>_L = \<L'_1[\Y]v_0,L_2[\Y]v_0\>_V = \<\Y L_1[\Y]v_0,L_2[\Y]v_0\>_V.$$
By Lemma \ref{antiauto;bilinear;W} and $\Y=\Y^{\dagger}$ from (\ref{X,Y;dagger}), 
it follows
$\<\Y L_1[\Y]v_0,L_2[\Y]v_0\>_V=\< L_1[\Y]v_0,\Y L_2[\Y]v_0\>_V$.
But the right-hand side in (\ref{bil;property;pf}) is equal to $\< L_1[\Y]v_0,\Y L_2[\Y]v_0\>_V $. The result follows.
\end{proof}

\noindent
We now show that the Laurent polynomials $\varepsilon^{+}_i, \varepsilon^{-}_i$ satisfy the orthogonality relation with respect to the bilinear form (\ref{bilinear;L}).
\begin{theorem}\label{orthogonality;s,r1,r2}
Let $\varepsilon^{+}_i, \varepsilon^{-}_i$ be the Laurent polynomials
in Definition {\rm \ref{e;s,r,h}}.
For $-D \leq r \leq D-1$, let $\y_r$ be an eigenvector of $\Y$ 
for the eigenvalue $\lambda_r$ as in Theorem {\rm \ref{e-vec;Y}}.
Then for $\sigma, \tau \in \{+,-\}$,
$$
\nonumber
\sum^{D-1}_{r=-D}\varepsilon^{\sigma}_i[\lambda_r]\overline{\varepsilon^{\tau}_j[\lambda_r]}\Vert \y_r\Vert^2
= \delta_{\sigma,\tau}\delta_{i,j}\Vert \hat{C}^{\sigma}_i \Vert^2,
$$
where $\Vert \y_r \Vert^2$ is given in Proposition {\rm \ref{norm;y}} 
and $\Vert \hat{C}^{\sigma}_i \Vert^2$ is given in Lemma {\rm \ref{norm;C;pm}}.
\end{theorem}
\begin{proof}
Using Lemma \ref{bilinear;L1,2} and Corollary \ref{Action;e;C} we find 
$$
\sum^{D-1}_{r=-D}\varepsilon^{\sigma}_i[\lambda_r]\overline{\varepsilon^{\tau}_j[\lambda_r]}\Vert \y_r\Vert^2
= \< \varepsilon^{\sigma}_i[\Y]v_0, \varepsilon^{\tau}_j[\Y]v_0\> _V
= \<\hat{C}^{\sigma}_i, \hat{C}^{\tau}_j\>_V.
$$
By Lemma \ref{norm;C;pm} the result follows.
\end{proof}
\noindent
Recall $\Vert \hat{C}^{\pm}_i\Vert^2$ from Lemma \ref{norm;C;pm} and the scalars $a,b,c,d$ from Definition \ref{abcd}.
Using (\ref{rels;abcd}) evaluate $\Vert \hat{C}^{\pm}_i\Vert^2$ in terms of $a,b,c,d$.
Then by Theorem \ref{orthogonality;s,r1,r2} it follows 
\begin{align}
\nonumber
\<\varepsilon^{\sigma}_i, \varepsilon^{\tau}_j\>_L   & = \delta_{\sigma,\tau}\delta_{i,j}\Vert \hat{C}^{\sigma}_i \Vert^2 \\
\label{norm;C;abcd}
& = 
\begin{cases}
\medskip
\delta_{\sigma,\tau}\delta_{i,j}\dfrac{(abq,ac,ad,abcd;q)_i}{a^{2i}(q,bc,bd,cd;q)_i} & \text{ if } \sigma=-, \\
\delta_{\sigma,\tau}\delta_{i,j}(-1)\dfrac{b(1-ac)(1-ad)}{a(1-bc)(1-bd)}\dfrac{(abq,acq,adq,abcd;q)_i}{a^{2i}(q,bcq,bdq,cd;q)_i} & \text{ if  } \sigma=+.
\end{cases}
\end{align}

\section{The algebra $\Hqi$}\label{Section;Hqi}

Recall from Section \ref{Section;UDAHA} that $\H$ is the universal
DAHA of type $(C^{\vee}_1,C_1)$. In this section we change $q$ by $q^{-1}$ and discuss the algebra $\Hqi$ and its module.
Recall from Section \ref{NonsymAWpolys} that $\tilde{\mathfrak{H}}=\tilde{\mathfrak{H}}(a,b,c,d;q)$ is 
the DAHA of type $(C^{\vee}_1,C_1)$. 
We will compare $\Hqi$ and $\tilde{\mathfrak{H}}$ shortly.

\begin{lemma}\label{iso;H;q;qinv}
There exists a $\C$-algebra isomorphism $\eta_1 : \Hqi \to \H$ that sends
$$
t_0 \mapsto t_0^{-1}, \qquad 
t_1 \mapsto t_0t_1^{-1}t_0^{-1}, \qquad 
t_2 \mapsto t_3^{-1}t_2^{-1}t_3, \qquad
t_3 \mapsto t_3^{-1}.
$$
Moreover, $\eta_1^2=1$.
\end{lemma}
\begin{proof}
Use Definition \ref{Def;UDAHA}.
\end{proof}
\begin{lemma}{\rm\cite[Lemma 16.8]{PT:DAHA}}\label{Lem16.8}
There is a surjective $\C$-algebra homomorphism $\eta_2 : \H \to \tilde{\mathfrak{H}}$ that sends
$$
t_0 \mapsto -(ab)^{-1/2}T_1, \quad t_1 \mapsto -(ab)^{1/2}T_1^{-1}Z^{-1}, \quad t_2 \mapsto -q^{-1}(cd)^{1/2}ZT_0^{-1}, \quad t_3 \mapsto -q^{1/2}(cd)^{-1/2}T_0.
$$
\end{lemma}

\noindent
Referring to Lemma \ref{iso;H;q;qinv} and Lemma \ref{Lem16.8} 
we define the map $\xi : \Hqi \to \tilde{\mathfrak{H}}$ to be the composition
$\xi=\eta_2\eta_1 $. Observe that this map is surjective and sends
\begin{equation}\label{map;xi}
\begin{split}
t_0 &\mapsto -(ab)^{1/2}T_1^{-1}, \\
t_1 &\mapsto -(ab)^{-1/2}T_1Z, \\
t_2 &\mapsto -q(cd)^{-1/2}Z^{-1}T_0, \\
t_3 &\mapsto -q^{-1/2}(cd)^{1/2}T_0^{-1}.
\end{split}
\end{equation}

\begin{lemma}\label{mapping;xi}
Recall $Y=T_1T_0 \in \tilde{\mathfrak{H}}$. 
Referring to the map $\xi$ in {\rm(\ref{map;xi})}, for $\X, \Y \in \Hqi$
\begin{align*}
& \X^{\xi}=q^{-1/2}(abcd)^{1/2}Y^{-1}, && (\X^{-1})^{\xi}=q^{1/2}(abcd)^{-1/2}Y, && \\
& \Y^{\xi}=Z, &&  (\Y^{-1})^{\xi}=Z^{-1}. &&
\end{align*}
\end{lemma}
\begin{proof}
Recall $\X=t_3t_0$ and $\Y=t_0t_1$. Use definition of $\xi$.
\end{proof}

In Section \ref{Section;UDAHA} we discussed the $\H$-module 
$\W(s,s^*,r_1,r_2,D;q)$. We consider an $\Hqi$-module on $\W$,
denoted by $\W(s',{s^*}',r'_1,r'_2,D';q^{-1})$. Note that $D=D'$.
By Proposition \ref{q-inv;Racah}, 
$$\W(s',{s^*}',r'_1,r'_2,D';q^{-1})=\W(s^{-1},s^{*-1},r^{-1}_1,r^{-1}_2,D;q^{-1}).$$
Recall the scalars $a$, $b$, $c$, $d$ from Definition \ref{abcd}. In the $q^{-1}$-Racah version, 
by (\ref{rels;abcd}) the scalars $a$, $b$, $c$, $d$ become 
$a^{-1}$, $b^{-1}$, $c^{-1}$, $d^{-1}$, respectively. 
By Lemma \ref{Hq-mod;abcd}, we can describe a structure of 
an $\Hqi$-module $\W(a^{-1}, b^{-1}, c^{-1}, d^{-1};q^{-1})$. 
For the rest of the paper, we denote
$\W_{q^{-1}} := \W(a^{-1}, b^{-1}, c^{-1}, d^{-1};q^{-1})$.

\begin{lemma}\label{Xinv-action;W'}
On the $\Hqi$-module $\W_{q^{-1}}$ the action of $\X^{-1}$ on $\{\hat{C}^{\pm}_i\}^{D-1}_{i=0}$ is as follows.
\begin{equation*}
\begin{cases}
\X^{-1}.\hat{C}^-_i = q^{i-\frac{1}{2}}(abcd)^{1/2}\hat{C}^-_i 
& \text{for} \quad i=0,1,\ldots, D-1,\\
\X^{-1}.\hat{C}^+_{i-1} = q^{-i+\frac{1}{2}}(abcd)^{-1/2}\hat{C}^+_{i-1} 
& \text{for} \quad i=1,2,\ldots, D.
\end{cases}
\end{equation*}
\end{lemma}
\begin{proof}
Use Lemma \ref{X-action;abcd}.
\end{proof}

\noindent
Recall $\varepsilon^-_i=\varepsilon^-_i[y;a,b,c,d \mid q]$
and $\varepsilon^+_i=\varepsilon^+_i[y;a,b,c,d \mid q]$
from Proposition \ref{e;abcd}.
We describe these polynomials in $q^{-1}$-version
$$
\varepsilon^{\sigma}_i[y ; q^{-1}]=\varepsilon^{\sigma}_i[y;a^{-1},b^{-1},c^{-1},d^{-1} \mid q^{-1}],
$$
where $0 \leq i \leq D-1$ and $\sigma \in \{+,-\}$.
To this end we need a few preliminary lemmas.
\begin{lemma}\label{(x;q)n;inverse}
For $n=0,1,2\ldots$ and a nonzero scalar $x \in \C$,
\begin{equation*}
(x^{-1};q^{-1})_n = (-1)^nx^{-n}q^{-\frac{n(n-1)}{2}}(x;q)_n.
\end{equation*}
\end{lemma}
\begin{proof}
Use the definition (\ref{(a;q)n}) to compute $(x^{-1};q^{-1})_n$. 
\end{proof}

\begin{lemma}\label{rel;P;Pqinv}
Let the monic polynomials $P_i[y;a,b,c,d\mid q]$ be  as in {\rm(\ref{monic;P;Pp})}. Then
$$
P_i[y;a,b,c,d\mid q] = P_i[y;a^{-1}, b^{-1},c^{-1},d^{-1} \mid q^{-1}].
$$
\end{lemma}
\begin{proof}
Evaluate $P_i[y;a^{-1}, b^{-1},c^{-1},d^{-1} \mid q^{-1}]$ using (\ref{monic;P;Pp})
and Lemma \ref{(x;q)n;inverse}.
\end{proof}

\noindent
For notational convenience, for $1 \leq i \leq D-1$ we define the  Laurent polynomials
$Q_i = Q_i[y;a,b,c,d \mid q]$ by
\begin{equation}\label{Q}
Q_i := a^{-1}b^{-1}y^{-1}(1-ay)(1-by)P^\p_{i-1},
\end{equation}
where we recall $P^\p_i = P_i[y;aq,bq,c,d \mid q]$ from (\ref{monic;P;Pp}); cf. (\ref{nonsym;Q}).

\begin{proposition}\label{e;abcd;qinv}
Let the scalars $a,b,c,d$ be as in Definition {\rm\ref{abcd}}.
Recall the polynomials $P_i=P_i[y;a,b,c,d \mid q]$ from {\rm (\ref{monic;P;Pp})} and $Q_i=Q_i[y;a,b,c,d\mid q]$ from {\rm(\ref{Q})}.
Then the Laurent polynomials $\varepsilon^{\pm}_i[y;q^{-1}]$
are described as follows: for $1 \leq i \leq D-1$\\
\begin{align*}
& \varepsilon^+_{i-1}[y ; q^{-1}]  =\frac{ab(1-q^i)(1-cdq^{i-1})}{(ab-1)} \frac{(abcd;q)_{2i-1}}{a^i(q,bc,bd,cd;q)_i}\left(P_i - Q_i\right),\\
& \varepsilon^-_{i}[y ; q^{-1}] = \frac{(1-abq^i)(1-abcdq^{i-1})}{(1-ab)}\frac{(abcd;q)_{2i-1}}{a^i(q,bc,bd,cd;q)_i}\left(P_i - \frac{ab(1-q^i)(1-cdq^{i-1})}{(1-abq^i)(1-abcdq^{i-1})}
Q_i\right).
\end{align*}
Moreover, 
$$
\varepsilon^-_0[y;q^{-1}] = 1, \qquad \qquad 
\varepsilon^+_{D-1}[y;q^{-1}] =  \frac{(abcd;q)_{2D-1}}{a^D(q,bc,bd;q)_D(cd;q)_{D-1}}P_D.
$$
\end{proposition}
\begin{proof}
In Proposition \ref{e;abcd}, replace $a,b,c,d,q$ by 
$a^{-1},b^{-1},c^{-1},d^{-1},q^{-1}$. 
Evaluate this using Lemma \ref{(x;q)n;inverse} 
and Lemma \ref{rel;P;Pqinv} and simplify the result.
Note that
$y(1-a^{-1}y^{-1})(1-b^{-1}qy^{-1})=a^{-1}b^{-1}y^{-1}(1-ay)(1-by)$. 
The results routinely follow.
\end{proof}

We finish this section with some comments.

\begin{lemma}\label{norm;C;qinv}
Referring the basis $\{\hat{C}^{\pm}_i\}^{D-1}_{i=0}$ for $\W_{q^{-1}}$,
the following hold. For $0 \leq i \leq D-1$,
\begin{align*}
&\Vert \hat{C}^-_i \Vert^2 = \frac{1-abq^i}{1-ab}\frac{(ab,ac,ad,abcd;q)_i}{a^{2i}(q,bc,bd,cd;q)_i}, \\
&\Vert \hat{C}^+_{i} \Vert^2 =\frac{ab(1-q^{i+1})(1-cdq^{i})}{(ab-1)(1-abcdq^{i})}\frac{(ab,ac,ad,abcd;q)_{i+1}}{a^{2i+2}(q,bc,bd,cd;q)_{i+1}}.
\end{align*}
\end{lemma}

\begin{proof} 
In line (\ref{norm;C;abcd}) we expressed $\Vert \hat{C}^{\pm}_i \Vert^2$ 
in terms of the scalars $a,b,c,d,q$. 
Replace these scalars by $a^{-1},b^{-1},c^{-1},d^{-1},q^{-1}$, respectively.
Evaluate this using Lemma \ref{(x;q)n;inverse} and simplify it. 
The result routinely follows.
\end{proof}

\noindent
\begin{note}\label{Lq-inv}
In Remark \ref{H-mod;L} we discussed the $\H$-module $L=L(a,b,c,d;q)$.
We recall from (\ref{bilinear;L}) that $\<\cdot, \cdot\>_L$ is the bilinear form on $L$.
Consider the $\Hqi$-module $L(a^{-1},b^{-1},c^{-1},d^{-1};q^{-1})$.
Abbreviate $L_{q^{-1}} := L(a^{-1},b^{-1},c^{-1},d^{-1};q^{-1})$.
Observe that the Laurent polynomials $\varepsilon^+[y;q^{-1}]$, $\varepsilon^-[y;q^{-1}]$ in Proposition \ref{e;abcd;qinv} form a basis for $L_{q^{-1}}$.
By (\ref{norm;C;abcd}) and Lemma \ref{norm;C;qinv}, the bilinear form $\<\cdot, \cdot\>_{L_{q^{-1}}}$ satisfies
\begin{align}\label{Lq_inv;bilinear;e;pm}
\<\varepsilon^{\sigma}_i[y;q^{-1}], \varepsilon^{\tau}_j[y;q^{-1}]\>_{L_{q^{-1}}}   
& = 
\begin{cases}
\medskip
\delta_{\sigma,\tau}\delta_{i,j}\dfrac{1-abq^i}{1-ab}\dfrac{(ab,ac,ad,abcd;q)_i}{a^{2i}(q,bc,bd,cd;q)_i}& \text{ if } \sigma=-, \\
\delta_{\sigma,\tau}\delta_{i,j}\dfrac{ab(1-q^{i+1})(1-cdq^{i})}{(ab-1)(1-abcdq^{i})}\dfrac{(ab,ac,ad,abcd;q)_{i+1}}{a^{2i+2}(q,bc,bd,cd;q)_{i+1}} & \text{ if  } \sigma=+.
\end{cases}
\end{align}
\end{note}

\section{Nonsymmetric Askey-Wilson polynomials and $\varepsilon^{\pm}_i$}\label{Rels;E;epsilon}

We continue to work with the algebra $\Hqi$ in Section \ref{Section;Hqi}.
Throughout this section we let the scalars $a,b,c,d$ be as in Definition \ref{abcd}.
Recall the nonsymmetric Laurent polynomials 
$\varepsilon^{+}_i=\varepsilon^{+}_i[y;q^{-1}]$, 
$\varepsilon^{-}_i=\varepsilon^{-}_i[y;q^{-1}]$ from Proposition \ref{e;abcd;qinv}.
Referring to this proposition,
we make a definition of the Laurent polynomials $E_i$ $(-D \leq i \leq D-1)$
which is a natural normalization of $\varepsilon^{+}_i, \varepsilon^{-}_i$.
\begin{definition}\label{E;finite_version}
Recall the Laurent polynomial sequences $P_i$ from (\ref{monic;P;Pp})
and $Q_i$ from (\ref{Q}).
With reference to Proposition \ref{e;abcd;qinv}, we define
\begin{align*}
& E_{-i} := {P}_i - {Q}_i && \hspace{-2cm}(1 \leq i \leq D-1), \\
& E_i := P_i - \frac{ab(1-q^i)(1-cdq^{i-1})}{(1-abq^i)(1-abcdq^{i-1})}Q_i
&& \hspace{-2cm}(1 \leq i \leq D-1), 
\end{align*}
and $E_0 := 1$ and $E_{-D} := P_D$; cf. Definition \ref{Koornwinder;NonsymAW}.
\end{definition}
\noindent
By Proposition \ref{e;abcd;qinv} and Definition \ref{E;finite_version} 
one can readily find that 
for $1 \leq i \leq D-1$,
\begin{align}
\label{e;abcd;qinv(1)}
& E_{-i}  =\frac{(ab-1)}{ab(1-q^i)(1-cdq^{i-1})} \frac{a^i(q,bc,bd,cd;q)_i}{(abcd;q)_{2i-1}}\varepsilon^+_{i-1}, \\
\label{e;abcd;qinv(2)}
& E_i  = \frac{(1-ab)}{(1-abq^i)(1-abcdq^{i-1})}\frac{a^i(q,bc,bd,cd;q)_i}{(abcd;q)_{2i-1}}\varepsilon^-_{i} .
\end{align}
Moreover,
\begin{equation}\label{e;abcd;qinv(3)}
E_0 = \varepsilon^-_0 , \qquad \qquad 
E_{-D} =  \frac{a^D(q,bc,bd;q)_D(cd;q)_{D-1}}{(abcd;q)_{2D-1}}\varepsilon^+_{D-1}.
\end{equation}

Recall from Note {\rm \ref{Lq-inv}} that 
$L_{q^{-1}}$ is the $\Hqi$-module and 
$\< \cdot, \cdot \>_{L_{q^{-1}}}$ is the bilinear form on $L_{q^{-1}}$.
By a comment in Note \ref{Lq-inv} and (\ref{e;abcd;qinv(1)})--(\ref{e;abcd;qinv(3)})
the Laurent polynomials $E_i$ in Definition \ref{E;finite_version} 
form a basis for $L_{q^{-1}}$. 
Observe that $E_i$ 
are orthogonal with respect to $\< \cdot, \cdot \>_{L_{q^{-1}}}$.
In the following proposition, we give a discrete version of Lemma \ref{norm;E}.

\begin{proposition}{\rm(cf. Lemma \ref{norm;E})}\label{finite.version;Lem2.2}
For $1 \leq i \leq D-1$
\begin{align}
\label{E;e;bilinear(1)}
 \<E_{-i}, E_{-i}\>_{L_{q^{-1}}} 
&= \frac{(ab-1)(1-abcdq^{2i-1})}{ab(1-q^i)(1-cdq^{i-1})}\frac{(q,ab,ac,ad,bc,bd,cd;q)_{i}}{(abcd;q)_{2i}(abcdq^{i-1};q)_{i}},\\
\label{E;e;bilinear(2)}
 \<E_i,E_i\>_{L_{q^{-1}}}  
&= \frac{(1-ab)(1-abcdq^{2i-1})}{(1-abq^i)(1-abcdq^{i-1})}\frac{(q,ab,ac,ad,bc,bd,cd;q)_i}{(abcd;q)_{2i}(abcdq^{i-1};q)_i}.
\end{align}
Moreover, 
\begin{equation}\label{E;e;bilinear(3)} 
\<E_0, E_0\>_{L_{q^{-1}}} = 1, \qquad 
\<E_{-D},E_{-D}\>_{L_{q^{-1}}} = 
\frac{ab(1-q^D)}{(ab-1)}\frac{(q,ab,ac,ad,bc,bd;q)_D(cd;q)_{D-1}}{(abcd;q)_{2D-1}(abcdq^{D-1};q)_D}.
\end{equation}
\end{proposition}
\begin{proof}
We first show (\ref{E;e;bilinear(1)}).
Evaluate the left-hand side of (\ref{E;e;bilinear(1)}) using
(\ref{e;abcd;qinv(1)}) and (\ref{Lq_inv;bilinear;e;pm}).
Simplify the result to get the right-hand side of (\ref{E;e;bilinear(1)}).
Lines (\ref{E;e;bilinear(2)}), (\ref{E;e;bilinear(3)}) are similarly obtained using (\ref{e;abcd;qinv(2)}), (\ref{e;abcd;qinv(3)}) togather with (\ref{Lq_inv;bilinear;e;pm}).
\end{proof}

From Proposition \ref{finite.version;Lem2.2}, 
we can view that the Laurent polynomials $E_i$ in Definition \ref{E;finite_version}
are a discrete analogue of the nonsymmetric Askey-Wilson 
polynomials in Definition \ref{Koornwinder;NonsymAW}. 
We further describe a discrete analogue of the eigenspace of $Y$ 
of the basic representation in Section \ref{NonsymAWpolys}.
Consider the $\Hqi$-module $L_{q^{-1}}$.
By construction each basis element $E_i$ is the eigenvector
of the action of $\X$ on $L_{q^{-1}}$.
With reference to Lemma \ref{Xinv-action;W'}, we visualize the eigenspaces of 
$\X^{-1}$ on $L_{q^{-1}}$ as follows.
Let $\eta=a^{-1/2}b^{-1/2}c^{-1/2}d^{-1/2}$.

\begin{center} 
\scalemath{0.65}{
\begin{tikzpicture}
  [scale=.8,thick,auto=left, every node/.style={circle}] 
  \node[fill=black,label=right:{\Large$q^{-1/2}\eta^{-1}, E_0$}] (n1) at (0,0) {};
  \node[draw, label=left:{\Large$q^{-1/2}\eta,E_{-1}$}] (n2) at (0,2)  {};
  \node[fill=black,label=right:{\Large$q^{1/2}\eta^{-1},E_1$}] (n3) at (2,2)  {};
  \node[draw,label=left:{\Large$q^{-3/2}\eta,E_{-2}$}] (n4) at (2,4) {};
  \node[fill=black,label=right:{\Large$q^{3/2}\eta^{-1},E_2$}] (n5) at (4,4)  {};
  \node[draw,label=left:{\Large$q^{-5/2}\eta,E_{-3}$}] (n6) at (4,6)  {};
  \node[fill=black,label=right:{\Large$q^{5/2}\eta^{-1},E_3$}] (n7) at (6,6)  {};
  \node[draw,label=left:{\Large$q^{-7/2}\eta, E_{-4}$}] (n8) at (6,8)  {};
  \node[fill=black,label=right:{\Large$q^{7/2}\eta^{-1}, E_4$}] (n9) at (8,8)  {};
  \node[draw,label=left:{\Large$$}] (n10) at (8,10)  {};
  \node[label=right:{\Large$$}] (n11) at (10,10)  {};  

  \foreach \from/\to in 
  {n1/n2,n2/n3,n3/n4,n4/n5,n5/n6,n6/n7,n7/n8,n8/n9}
      \draw (\from) -- (\to);
      \draw (n1) -- (n2) ;
      \draw (n3) -- (n4) ;
      \draw (n5) -- (n6) ;
      \draw (n7) -- (n8) ;
      \draw (n9) -- (n10) [dashed, ];
      \draw (n10) -- (n11) [dashed];
            
\end{tikzpicture}}\\
{Figure 3 : The eigenspaces of $\X^{-1}$}
\end{center}
Note that each white node represents the eigenspace of $\X^{-1}$ corresponding 
the eigenvalue $q^{-i+\frac{1}{2}}\eta$ and the eigenvector $E_{-i}$ for 
$i=1,2,\ldots, D$, and each black node represents the eigenspace of $\X^{-1}$ 
corresponding the eigenvalue $q^{i-\frac{1}{2}}\eta^{-1}$ and the eigenvector 
$E_i$ for $i=0,1,2,\ldots, D-1$. 
Compare Figure 3 with Figure 1 in Section \ref{NonsymAWpolys}.
Then one can see that Figure 3 coincides with Figure 1 
up to $(D-1)$-th horizontal edge. 
This is very natural since $\X^{-1}$ corresponds to $q^{1/2}\eta Y$ 
by Lemma \ref{mapping;xi}.

\section{Conclusion}

 In this paper we have studied certain Laurent polynomials, which are naturally obtained from a $Q$-polynomial distance-regular graph $\Ga$ of $q$-Racah type that contains a Delsarte clique.
Using an irreducible module of the universal DAHA $\H$ and $\Ga$, we proved the orthogonality relations for those polynomials.
Consequently, we showed that the above Laurent polynomials are a finite/ combinatorial analogue of the nonsymmetric Askey-Wilson polynomials.

\medskip
As we mentioned earlier in Section \ref{Intro}, 
the theorem of D. Leonard \cite{DL} (cf. \cite[Section~III.5]{BI}) characterized the terminating branch
of the Askey scheme \cite{KLS} of basic hypergeometric orthogonal polynomials
by the duality property of $\Ga$,
which has had a significant impact on the theory of orthogonal polynomials.
According to this theorem, the $q$-Racah polynomials are the most general self-dual
orthogonal polynomials in the above branch. 
Our results in the present paper can be thought of as a nonsymmetric version
of the $q$-Racah polynomials in the situation of Leonard's theorem.
We are planning to apply our results to the study of nonsymmetric version of 
other types of orthogonal polynomials in the terminating branch of the Askey scheme, such as Krawtchouk polynomials, using $\Ga$ of the corresponding type that contains a Delsarte clique. This will give a characterization of a nonsymmetric case of Leonard's theorem.

\section{Appendix}

In this Appendix we describe the $\H$-module structure
$\W(s,s^*,r_1,r_2,D;q)$ twisted via $\rho$ (see \S \ref{Section;UDAHA}) 
and display the action of $\Y$
on this module explicitly.
Recall the scalars $s, s^*, r_1, r_2, D, q$ from Note \ref{Note;S4}.

\subsection{An $\H$-module in terms of the scalars $s, s^*, r_1, r_2, D, q$.}\label{Apdx;H-mod}

\begin{definition}\label{H-mod;s,r,h}
(cf. \cite[Definition 11.1]{JHL})\label{H-module}
We define some matrices as follows.
\begin{itemize}
\item[(a)] For $1 \leq i \leq D-1$, the $(2 \times 2)$-matrix $t_0(i)$ is 
$$
\begin{bmatrix}
\frac{q^{D/2}(1-q^{i-D})(1-s^*q^{i+1})}{1-s^*q^{2i+1}} + \frac{1}{q^{D/2}} & &
\frac{q^{D/2}(q^{i-D}-1)(1-s^*q^{i+1})}{1-s^*q^{2i+1}} \vspace{0.3cm}\\
\frac{(1-q^i)(1-s^*q^{D+i+1})}{q^{D/2}(1-s^*q^{2i+1})} & &
\frac{(q^i-1)(1-s^*q^{D+i+1})}{q^{D/2}(1-s^*q^{2i+1})}+\frac{1}{q^{D/2}}
\end{bmatrix}
$$
and 
$$
t_0(0) = \left[ ~ \frac{1}{q^{D/2}}~\right], \qquad \qquad 
t_0(D) = \left[ ~ \frac{1}{q^{D/2}}~\right].
$$
\item[(b)] For $0 \leq i \leq D-1$, the $(2 \times 2)$-matrix $t_1(i)$ is 
$$
\begin{bmatrix}
\frac{1}{(s^*r_1r_2)^{1/2}}\left( \frac{(r_1-s^*q^{i+1})(r_2-s^*q^{i+1})}{1-s^*q^{2i+2}} + s^* \right) &&
-\left(\frac{s^*}{r_1r_2}\right)^{1/2}\frac{(1-r_1q^{i+1})(1-r_2q^{i+1})}{1-s^*q^{2i+2}}  \vspace{0.3cm}\\
\frac{1}{(s^*r_1r_2)^{1/2}} \frac{(r_1-s^*q^{i+1})(r_2-s^*q^{i+1})}{1-s^*q^{2i+2}} &&
\left(\frac{s^*}{r_1r_2}\right)^{1/2}\left(1-\frac{(1-r_1q^{i+1})(1-r_2q^{i+1})}{1-s^*q^{2i+2}}\right)
\end{bmatrix}.
$$
\item[(c)] $0 \leq i \leq D-1$, the $(2 \times 2)$-matrix $t_2(i)$ is 
$$
\begin{bmatrix}
\frac{1}{q^{i+1}(r_1r_2)^{1/2}}\left( 1-\frac{(1-r_1q^{i+1})(1-r_2q^{i+1})}{1-s^*q^{2i+2}} \right) & &
\frac{s^*q^{i+1}}{(r_1r_2)^{1/2}} \frac{(1-r_1q^{i+1})(1-r_2q^{i+1})}{1-s^*q^{2i+2}} \vspace{0.3cm}\\
-\frac{1}{s^*q^{i+1}(r_1r_2)^{1/2}}\frac{(r_1-s^*q^{i+1})(r_2-s^*q^{i+1})}{1-s^*q^{2i+2}} & &
\frac{q^{i+1}}{(r_1r_2)^{1/2}} \left( \frac{(r_1-s^*q^{i+1})(r_2-s^*q^{i+1})}{1-s^*q^{2i+2}} + s^*\right)
\end{bmatrix}.
$$
\item[(d)] For $1 \leq i \leq D-1$, the $(2 \times 2)$-matrix $t_3(i)$ is
$$
\begin{bmatrix}
\frac{1}{q^i(s^*q)^{1/2}} \left( \frac{(q^i-1)(1-s^*q^{D+i+1})}{q^{D/2}(1-s^*q^{2i+1})} + \frac{1}{q^{D/2}} \right) & & \frac{1}{q^i(s^*q)^{1/2}} \left( \frac{q^{D/2}(1-q^{i-D})(1-s^*q^{i+1})}{1-s^*q^{2i+1}} \right) \vspace{0.3cm}\\
{q^i(s^*q)^{1/2}} \left( \frac{(q^i-1)(1-s^*q^{D+i+1})}{q^{D/2}(1-s^*q^{2i+1})} \right) & &
{q^i(s^*q)^{1/2}} \left( \frac{q^{D/2}(1-q^{i-D})(1-s^*q^{i+1})}{1-s^*q^{2i+1}} + \frac{1}{q^{D/2}} \right)
\end{bmatrix}
$$ 
and
$$
t_3(0) = \left[ ~ (s^*q^{D+1})^{1/2}~\right], \qquad \qquad 
t_3(D) = \left[ ~ \tfrac{1}{(s^*q^{D+1})^{1/2}}~\right].
$$
\end{itemize}
\end{definition}

\noindent
With reference to Definition \ref{H-mod;s,r,h}, we define a $2D\times 2D$ block
diagonal matrix $\mcal{T}_n (n \in \I)$ as follows.
\begin{align*}
\mcal{T}_0 & :=  \bdiag\Big(t_0(0), t_0(1),\ldots, t_0(D-1), t_0(D)\Big),\\
\mcal{T}_1 & :=  \bdiag\Big(t_1(0), t_1(1),\ldots, t_1(D-1) \Big),\\
\mcal{T}_2 & :=  \bdiag\Big(t_2(0), t_2(1),\ldots, t_2(D-1) \Big),\\
\mcal{T}_3 & :=  \bdiag\Big(t_0(0), t_0(1),\ldots, t_0(D-1), t_0(D)\Big).
\end{align*}
One checks that (i) each $\mcal{T}_n (n \in \I)$ is invertible; 
(ii) $\mcal{T}_n+\mcal{T}_n^{-1} = (\k_n +\k_n)I$, where $\k_n$ is from (\ref{k_n});
(iii) $\mcal{T}_0\mcal{T}_1\mcal{T}_2\mcal{T}_3 = q^{-1/2}I$.
By this and Definition \ref{Def;UDAHA}, 
there exists an $\H$-module structure on $\W$ such that
for $n\in \I$ the matrix $\mcal{T}_n$ represents the generator $t_n$
relative to the basis $\{\hat{C}^{\pm}_i\}^{D-1}_{i=0}$ \cite[Proposition 11.10]{JHL}.

\subsection{the action of $\Y$}\label{Action;Y}

Referring to Section \ref{Apdx;H-mod}, we display
the action of $\Y$ on $\{\hat{C}^{\pm}_i\}^D_{i-1}$. For this section, we abbreviate 
$\W_q := \W(s,s^*,r_1,r_2,D;q)$.

\begin{lemma}\label{Y-action}{\rm (cf. \cite[Lemma 12.2]{JHL})} On $\W_q$,
\begin{enumerate}
\item[\rm(a)] For $0 \leq i \leq D-1$, the action $\Y.\hat{C}^-_i$ is given as a linear combination with the following terms and coefficients:
\begin{align*}
\begin{tabular}{l | l}
\text{\rm term} & \qquad \qquad \text{\rm coefficient} \\
\hline \hline
$\hat{C}^+_{i-1}$ & $\left( \frac{q^D}{s^*r_1r_2} \right)^{1/2} \left( \frac{(r_1-s^*q^{i+1})(r_2-s^*q^{i+1})}{1-s^*q^{2i+2}} + s^* \right)\left( \frac{(q^{i-D}-1)(1-s^*q^{i+1})}{1-s^*q^{2i+1}} \right) $ \\
\vspace{-0.3cm} ~ & ~\\
$\hat{C}^-_i$ & $\left( \frac{1}{q^Ds^*r_1r_2} \right)^{1/2} \left( \frac{(r_1-s^*q^{i+1})(r_2-s^*q^{i+1})}{1-s^*q^{2i+2}} + s^* \right)\left( \frac{(q^i-1)(1-s^*q^{D+i+1})}{1-s^*q^{2i+1}} +1 \right)$ \\
\vspace{-0.3cm} ~ & ~\\
$\hat{C}^{+}_i$ & $\left( \frac{q^D}{s^*r_1r_2} \right)^{1/2}\left(\frac{(r_1-s^*q^{i+1})(r_2-s^*q^{i+1})}{1-s^*q^{2i+2}}\right)\left( \frac{(1-q^{i+1-D})(1-s^*q^{i+2})}{1-s^*q^{2i+3}} + \frac{1}{q^D} \right)$\\
\vspace{-0.3cm} ~ & ~\\
$\hat{C}^{-}_{i+1}$ & $\left(\frac{1}{s^*r_1r_2q^D}\right)^{1/2}\frac{(1-q^{i+1})(1-s^*q^{D+i+2})(r_1-s^*q^{i+1})(r_2-s^*q^{i+1})}{(1-s^*q^{2i+2})(1-s^*q^{2i+3})}$ \\
\end{tabular},
\end{align*}
where $\hat{C}_{-1}^{+}=0$ and $\hat{C}^-_{D}=0$.

\item[\rm(b)] For $0 \leq i \leq D-1$, the action $\Y.\hat{C}^+_i$ is given as a linear combination with the following terms and coefficients:
\begin{align*}
\begin{tabular}{l | l}
\text{\rm term} & \qquad \qquad \text{\rm coefficient} \\
\hline \hline
$\hat{C}^+_{i-1}$ & $\left( \frac{s^*q^D}{r_1r_2} \right)^{1/2} \frac{(1-q^{i-D})(1-s^*q^{i+1})(1-r_1q^{i+1})(1-r_2q^{i+1})}{(1-s^*q^{2i+1})(1-s^*q^{2i+2})}$ \\
\vspace{-0.3cm} ~ & ~\\
$\hat{C}^-_i$ &  $-\left( \frac{s^*}{r_1r_2q^D} \right)^{1/2} \left(\frac{(1-r_1q^{i+1})(1-r_2q^{i+1})}{1-s^*q^{2i+2}}\right) \left( \frac{(q^i-1)(1-s^*q^{i+D+1})}{1-s^*q^{2i+1}} + 1\right)$\\
\vspace{-0.3cm} ~ & ~\\
$\hat{C}^{+}_i$ &  $\left( \frac{s^*q^D}{r_1r_2} \right)^{1/2} \left( 1-\frac{(1-r_1q^{i+1})(1-r_2q^{i+1})}{1-s^*q^{2i+2}} \right) \left( \frac{(1-q^{i+1-D})(1-s^*q^{i+2})}{1-s^*q^{2i+3}}+\frac{1}{q^D}\right)$ \\
\vspace{-0.3cm} ~ & ~\\
$\hat{C}^{-}_{i+1}$ & $\left( \frac{s^*}{r_1r_2q^D} \right)^{1/2} \left( 1 - \frac{(1-r_1q^{i+1})(1-r_2q^{i+1})}{1-s^*q^{2i+2}} \right)\left( \frac{(1-q^{i+1})(1-s^*q^{D+i+2})}{1-s^*q^{2i+3}} \right)$ \\
\end{tabular},
\end{align*}
where $\hat{C}_{-1}^{+}=0$ and $\hat{C}^-_{D}=0$.
\end{enumerate}
\end{lemma}

\begin{lemma}\label{Yinv-action}{\rm (cf. \cite[Lemma 12.3]{JHL})}  On $\W_q$,
\begin{enumerate}
\item[\rm(a)] For $0 \leq i \leq D-1$, the action $\Y^{-1}.\hat{C}^-_i$ is given as a linear combination with the following terms and coefficients:
\begin{align*}
\begin{tabular}{l | l}
\text{\rm term} & \qquad \qquad \text{\rm coefficient} \\
\hline \hline
$\hat{C}^-_{i-1}$ &  $\left( \frac{s^*q^D}{r_1r_2} \right)^{1/2} \frac{(1-q^{i-D})(1-s^*q^{i+1})(1-r_1q^i)(1-r_2q^i)}{(1-s^*q^{2i})(1-s^*q^{2i+1})}$\\
\vspace{-0.3cm} ~ & ~  \\
$\hat{C}^+_{i-1}$ & $\left( \frac{q^D}{s^*r_1r_2} \right)^{1/2} \left( \frac{(1-q^{i-D})(1-s^*q^{i+1})}{1-s^*q^{2i+1}} \right)\left( \frac{(r_1-s^*q^{i})(r_2-s^*q^i)}{1-s^*q^{2i}} + s^* \right)$\\
\vspace{-0.3cm} ~ & ~\\
$\hat{C}^{-}_i$ & $\left( \frac{s^*q^D}{r_1r_2} \right)^{1/2} \left( \frac{(1-q^{i-D})(1-s^*q^{i+1})}{1-s^*q^{2i+1}} + \frac{1}{q^D}\right)\left( 1-\frac{(1-r_1q^{i+1})(1-r_2q^{i+1})}{1-s^*q^{2i+2}} \right)$\\
\vspace{-0.3cm} ~ & ~\\
$\hat{C}^{+}_{i}$ & $-\left( \frac{q^D}{s^*r_1r_2} \right)^{1/2} \left( \frac{(1-q^{i-D})(1-s^*q^{i+1})}{1-s^*q^{2i+1}} + \frac{1}{q^D} \right)\left( \frac{(r_1-s^*q^{i+1})(r_2-s^*q^{i+1})}{1-s^*q^{2i+2}} \right)$ \\
\end{tabular},
\end{align*}
where $\hat{C}_{-1}^{-}=0$ and $\hat{C}^+_{-1}=0$.

\item[\rm(b)] For $0 \leq i \leq D-1$, the action $\Y^{-1}.\hat{C}^+_i$ is given as a linear combination with the following terms and coefficients:
\begin{align*}
\begin{tabular}{l | l}
\text{\rm term} & \qquad \qquad \text{\rm coefficient} \\
\hline \hline
$\hat{C}^-_{i}$ &  $\left( \frac{q^D}{s^*r_1r_2} \right)^{1/2} \left( \frac{(q^{i+1}-1)(1-s^*q^{D+i+2})}{1-s^*q^{2i+3}}+1 \right)\left( \frac{(1-r_1q^{i+1})(1-r_2q^{i+1})}{1-s^*q^{2i+2}} \right)$\\
\vspace{-0.3cm} ~ & ~  \\
$\hat{C}^+_{i}$ & $\left( \frac{1}{s^*r_1r_2q^D} \right)^{1/2}\left(\frac{(q^{i+1}-1)(1-s^*q^{D+i+2})}{1-s^*q^{2i+3}}+1\right)\left( \frac{(r_1-s^*q^{i+1})(r_2-s^*q^{i+1})}{1-s^*q^{2i+2}}+s^* \right)$\\
\vspace{-0.3cm} ~ & ~\\
$\hat{C}^{-}_{i+1}$ & $\left(\frac{s^*}{r_1r_2q^D}\right)^{1/2} \left( \frac{(1-q^{i+1})(1-s^*q^{D+i+2})}{1-s^*q^{2i+3}} \right)\left( \frac{(1-r_1q^{i+2})(1-r_2q^{i+2})}{1-s^*q^{2i+4}} -1 \right)$\\
\vspace{-0.3cm} ~ & ~\\
$\hat{C}^{+}_{i+1}$ & $\left( \frac{1}{s^*r_1r_2q^D} \right)^{1/2} \frac{(1-q^{i+1})(1-s^*q^{D+i+2})(r_1-s^*q^{i+2})(r_2-s^*q^{i+2})}{(1-s^*q^{2i+3})(1-s^*q^{2i+4})}$\\
\end{tabular},
\end{align*}
where $\hat{C}_{-1}^{-}=0$ and $\hat{C}^+_{-1}=0$.
\end{enumerate}
\end{lemma}

\section{Acknowledgements}
The author would like to thank Paul Terwilliger and Hajime Tanaka 
for many helpful discussions and comments. 
In particular, Hajime Tanaka gave valuable ideas and suggestions
for Sections 8, 9.
The author also thanks the two anonymous referees for careful reading and
helpful comments.
This work is supported by the JSPS KAKENHI; grant numbers: $26\cdot04019$.


\bigskip \bigskip
{\footnotesize \sc Research Center for Pure and Applied Mathematics,
Graduate School of Information Sciences,
Tohoku University,
6-3-09 Aramaki-Aza-Aoba, Aoba-ku, Sendai 980-8579, Japan. 

{\it Email address}: {\tt jhlee@ims.is.tohoku.ac.jp }}


\begin{thebibliography}{10}

\bibitem{AW2} {R.~Askey and J.~Wilson},
	A set of orthogonal polynomials that generalize the Racah coefficients
	of $6-j$ symbols, {\it SIAM J. Math. Anal.}, 10:1008--1016, 1979.

\bibitem{AW} {R.~Askey and J.~Wilson},
	Some basic hypergeometric orthogonal polynomials that generalize Jacobi polynomials, {\it Mem. Amer. Math. Soc.} (1985), no. 319.

\bibitem{BI} {E. Bannai and T. Ito},
	Algebraic Combinatorics I: Association Schemes, Benjamin-Cummings Lecture Note Ser., vol. 58, The Benjamin/ Cumming Publishing Company, Inc., London, 1984.

\bibitem{BCN} {A. Brouwer, A. Cohen, and A. Neumaier},
	Distance-Regular Graphs, Springer-Verlag, Berlin, 1989.

\bibitem{TC}{T. S. Chihara},
	{\it  An introduction to orthogonal polynomials},
	Gordon and Breach, 1978.

\bibitem{DKT}{E.~R.~van~Dam, J.~H.~Koolen, and H.~Tanaka}, 
	Distance-regular graphs, Electron. J. Combin. (2016) \#DS22; \url{arXiv:1410.6294}.


\bibitem{PD} {P. Delsarte},
	An algebraic approach to the association schemes of coding theory,
	Philips Res. Rep. Suppl., No. 10 (1973).


\bibitem{CG}{C. D. Godsil},
	{Algebraic combinatorics},
	Chapman \& Hall, New York, 1993.

\bibitem{GR} {G. Gasper and M. Rahman},
	{\it Basic hypergeometric series}, 2nd edn., Cambridge University Press, 2004.

\bibitem{GST}{D. Gijswijt, A. Schrijver, and H. Tanaka},
	New upper bounds for nonbinary codes based on the Terwilliger algebra
	and semidefinite programming,
	J. Combin. Theory Ser. A 113 (2006) 1719--1731.



\bibitem{ITT}{T. Ito, K. Tanabe, and P. Terwilliger},
	Some algebra related to $P$- and $Q$-polynomial association schemes.
	Codes and association schemes (Piscataway, NJ, 1999),
	167--192, DIMACS Ser. Discrete Math. Theoret. Comput. Sci.,
	{\bf 56} Amer. Math. Soc., Providence, RI, 2001.

\bibitem{IT3}{T. Ito and P. Terwilliger},
	Tridiagonal pairs and the quantum affine algebra $U_q(\widehat{\mathfrak{sl}}_2)$,
	Ramanujan J. 13 (2007), no. 1-3, 39--62.
	
\bibitem{IT2}{T. Ito and P. Terwilliger},
	Distance-regular graphs of $q$-Racah type and the $q$-tetrahedron algebra,
	Michigan Math. J. 58 (2009), no. 1, 241--254.

\bibitem{IT4}{T. Ito and P. Terwilliger},
	The augmented tridiagonal algebra, 
	{\it Kyushu J. Math.} {\bf 64} (2010), 81--144.
	
\bibitem{IT}{T. Ito and P. Terwilliger},
	Double affine Hecke algebra of rank 1 and the $\mathbb{Z}_3$ 
	symmetric Askey-Wilson relations,
	{\it SIGMA} {\bf 6} (2010) 065, 9 pages.

\bibitem{KLS}{R. Koekoek, P. A. Lesky and R. F. Swarttouw}, 
	Hypergeometric orthogonal polynomials and their $q$-analogues, 
	Springer-Verlag, Berlin, 2010.

\bibitem{TK} {T. Koornwinder,}
	The Relationship between Zhedanov's Algebra $AW(3)$ and 
	the Double Affine Hecke Algebra in the Rank One Case,
	{\it SIGMA} {\bf 3} (2007), 063, 15 pages.

\bibitem{KB} {T. Koornwinder and F. Bouzeffour}, 
	Nonsymmetric Askey-Wilson polynomials as vector-valued polynomials, 
	{\it Appl. Anal.} {\bf 90} (2011), 731-746.

\bibitem{DL} {D. Leonard},
	Orthogonal polynomials, duality, and association schemes,
	{\it SIAM J. Math. Anal.} Vol. 13 no. 4 (1982), pp. 656--663.

\bibitem{JHL} {J.-H. Lee},
	$Q$-polynomial distance-regular graphs and a double affine Hecke algebra
	of rank one,
	{\it Linear Algebra Appl.} {\bf 439} (2013), 3184--3240.

\bibitem{IM}{I. Macdonald,}
	Affine Hecke algebra and orthogonal polynomials,
	Cambridge University Press, 2003.

	

\bibitem{NS}{M. Noumi and J. V. Stokman,}
	Askey-Wilson polynomials: an affine Hecke algebraic approach,
	in { Laredo Lectures on Orthogonal Polynomials and Special Functions},
	Nova Sci. Publ., Hauppauge, NY, 2004, pp. 111-144; 
	\url{arXiv:math/0001033v1} \texttt{[math.QA].}
	
\bibitem{SS}{S. Sahi},
	Nonsymmetric Koornwinder polynomials and duality,
	{\it Ann. of Math.} (2) {\bf 150} (1999), 267--282.

\bibitem{AS}{A.Schrijver},
	New code upper bounds from the Terwilliger algebra and semidefinite programming,
	IEEE Trans. Inform. Theory 51 (2005) 2859--2866.

\bibitem{HS} {H. Suzuki},
	The Terwilliger algebra associated with a set of vertices in a distance-regular graphs, 
	{\it J. Algebraic Combin.} {\bf 22} (2005), 5--38.

\bibitem{HT}{H. Tanaka},
	New proofs of the Assmus--Mattson theorem based on the Terwilliger
	algebra,
	{\it European J. Combin.} {\bf 30} (2009) 736--746.

\bibitem{PT:T-alg} {P. Terwilliger},
	The subconstituent algebra of an association scheme, I.,
	{\it J. Algebraic Combin.} {1} (1992), 363--388.

\bibitem{PT;T-alg;II} {P. Terwilliger},
	The subconstituent algebra of an association scheme, II.,
	{\it J. Algebraic Combin.} {2} (1993), 73--103.

\bibitem{PT;T-alg;III} {P. Terwilliger},
	The subconstituent algebra of an association scheme III.,
	{\it J. Algebraic Combin.} {2} (1993), 177--210.

\bibitem{PT;2lin} {P. Terwilliger}
	Two linear transformations each tridiagonal with respect to an eigenbasis
	of the other, {\it Linear Alg. Appl.} {\bf 330} (2001), 149--203

\bibitem{PT:LP_qPoly} {P. Terwilliger},
	Leonard pairs and the $q$-Racah polynomials,
	{\it Linear Algebra Appl.} {\bf 387} (2004), 235--276.
	
\bibitem{PT:Madrid} {P. Terwilliger,}
	Two linear transformations each tridiagonal with respect to an eigenbasis 
	of the other;
	an algebraic approach to the Askey scheme of orthogonal polynomials, 
	Lecture Notes for the Summer School on Orthogonal Polynomials 
	and Special Functions, Universidad Carlos III de Madrid, Leganes, Spain, 
	July 8--18, 2004, \url{arXiv:math.QA/0408390}.

\bibitem{PT:PA} {P. Terwilliger}
	Two linear transformations each tridiagonal with respect to an eigenbasis 
	of the others; comments on the parameter array,
	{\it Des. Codes Cryptogr.} {\bf 34} (2005), 307--332.


\bibitem{PT:DAHA} {P. Terwilliger,} 
	The universal Askey-Wilson algebra and DAHA of type $(C^{\vee}_1,C_1)$, 
	{\it SIGMA} {\bf 9} (2013), 047, 40 pages.

\end{thebibliography}
\end{document}